\documentclass{amsart}
\usepackage{amsfonts,amssymb,amsmath,amstext,amscd,epsfig}
\usepackage{amsthm,cases,color,graphicx,graphics}
\usepackage{dsfont}
\usepackage{mathtools}
\usepackage{enumitem}
\usepackage[usenames,dvipsnames,svgnames,table]{xcolor}
\usepackage[linktocpage=true,colorlinks=true,linkcolor=RubineRed!85!black,citecolor=ForestGreen,urlcolor=green,pdfborder={0 0 0}]{hyperref}

\theoremstyle{plain}
\newtheorem{theorem}{Theorem}[section]
\newtheorem{corollary}[theorem]{Corollary}
\newtheorem{proposition}[theorem]{Proposition}
\newtheorem{lemma}[theorem]{Lemma}
\newtheorem{remark}[theorem]{Remark}
\numberwithin{equation}{section}

\newcommand*{\dif}{\mathop{}\!\mathrm{d}}
\renewcommand{\Re}{\operatorname{Re}}
\renewcommand{\Im}{\operatorname{Im}}

\title{Averaging lemmas and hypoellipticity}
\author{Yuzhe Zhu} 
\date{February 24, 2025} 
\address{Department of Mathematics, University of Chicago, Chicago, Illinois 60637, USA}
\email{yuzhezhu@uchicago.edu}

\begin{document}
\begin{abstract}
We use the methods of commutator and fundamental solutions to establish averaging lemmas and hypoelliptic estimates for purely kinetic transport equations. Assuming certain amount of velocity regularity for solutions, we extend our analysis using the commutator method to derive the averaging and hypoelliptic regularity properties for kinetic equations in the presence of general inhomogeneous fluxes. These results find applications in the study of hypoelliptic advection-diffusion equations and kinetic formulations of hyperbolic conservation laws including Burgers' equation with transport and isentropic gas dynamics. 
\end{abstract}

\maketitle
\hypersetup{bookmarksdepth=2}
\setcounter{tocdepth}{1}
\tableofcontents

\section{Introduction}
Let the operator $D_{z}^m:=(-\Delta_{z})^{m/2}$ be the derivative with respect to the variable $z$ of order $m\in\mathbb{R}$. We study the regularity properties for the solution $f=f(t,x,v)$, with $(t,x,v)\in\mathbb{R}\times\mathbb{R}^d\times\mathbb{R}^d$ and $d\in\mathbb{N}_+$, to the kinetic transport equation 
\begin{align}\label{KE}
\partial_tf+\nabla_x\cdot\left(vf+uf\right)=\left(\lambda D_t^\alpha+D_x^\alpha\right) D_v^\beta g, 
\end{align}
where the vector field $u=u(t,x):\mathbb{R}\times\mathbb{R}^d\rightarrow\mathbb{R}^d$, the scalar function $g=g(t,x,v)$, and the constants $\lambda,\alpha\in[0,1]$, $\beta\ge0$. 

The equation usually models the behaviour of a system of particles or emerges from various kinetic formulations; see Section~\ref{apps} below. The solution $f(t,x,v)$ can be interpreted as a distribution of particles at time $t$ that occupy the position $x$ and have a state described by the velocity variable $v$. The averaging lemmas for purely kinetic equations of the type \eqref{KE} without the macroscopic velocity field $u$, established in early works such as \cite{Ago,GPS,GLPS,Gerard,DPLM}, seek to understand the regularizing effect on the velocity average (or moment) of the solution $f$, that is, 
\begin{align}\label{af}
\rho_\phi[f](t,x) := \int_{\mathbb{R}^d}f(t,x,v)\phi(v)\dif v,\quad \phi\in C_c^\infty(\mathbb{R}^d). 
\end{align}
Combining this with additional regularity of solutions with respect to the velocity variable, hypoelliptic regularization for the solution itself was further explored in \cite{GG,Bouchut}. 

\subsection{Main results}\label{main-thm}
The objective of this work is to use the hypoelliptic methods, including the commutator and energy methods developed in \cite{Bouchut,JLT,JLTs,AL} and the method of fundamental solutions applied in \cite{PP,IS,GM}, to quantify the  regularity properties for solutions to \eqref{KE}. 

We first give the regularity results for solutions to \eqref{KE} in the purely kinetic case. Let us recall that $B_R$ denotes the Euclidean ball in $\mathbb{R}^d$ centered at the origin with radius $R$, and the constants $p_\star$ and $q_\star$ are the conjugate exponents to $p$ and $q$, respectively, that is, 
\begin{align*}
1/p+1/{p_\star}=1/q+1/{q_\star}=1. 
\end{align*}

\begin{theorem}\label{reg-0} 
Let the constants $\lambda,\alpha\in[0,1]$, $\beta\ge0$, $p,q\in(1,\infty)$, $R\ge1$. Let $f$ be a solution to \eqref{KE} with $u=0$, $g\in L^{p}_{t,x}L_v^{q}(\mathbb{R}^{1+2d})$, and its support ${\rm supp\;\!} f:=\{(t,x,v):f(t,x,v)\neq0\}\subset\mathbb{R}\times\mathbb{R}^d\times B_R$. 
\begin{enumerate}[label=(\roman*),leftmargin=0.675cm]
\item\label{reg-1}
If $f\in L_{t,x}^2H_v^k(\mathbb{R}^{1+2d})$ satisfies $D_v^lf\in L_{t,x}^{p_\star}L_v^{q_\star}(\mathbb{R}^{1+2d})$ for some $k\ge0$ and $0\le l\le\beta$, then for 
\begin{align*}
s:=\frac{(1-\alpha)(1+2k)}{2(1+\beta+2k-l)}, 
\end{align*}
we have $f\in H_{t,x}^sH_v^{_{-3/2}}(\mathbb{R}^{1+2d})$ with  
\begin{align*}
\|D_x^sf\|_{L_{t,x}^2H_v^{{-3/2}}}^2 + R^{-2s}\|D_t^sf\|_{L_{t,x}^2H_v^{{-3/2}}}^2\\
\le C\left(1+\lambda R^\alpha\right) \|D_v^lf\|_{L_{t,x}^{p_\star}L_v^{q_\star}} \|g\|_{L_{t,x}^{p}L_v^{q}} + C\|f\|_{L_{t,x}^2H_v^k}^2&. 
\end{align*} 
\item\label{reg-2}
If $f\in L_{t,x}^2H_v^k(\mathbb{R}^{1+2d})$ satisfies $D_v^lf\in L_{t,x}^{p_\star}L_v^{q_\star}(\mathbb{R}^{1+2d})$ for some $k\ge0$ and $0\le l\le1+\beta$, then for 
\begin{align*}
s:=\frac{(1-\alpha)k}{1+\beta+2k-l},
\end{align*}
we have $f\in L_v^2H_{t,x}^s(\mathbb{R}^{1+2d})$ with  
\begin{align*}	
\|D_x^sf\|_{L_{t,x,v}^2}^2 + R^{-2s}\|D_t^sf\|_{L_{t,x,v}^2}^2\\
\le C\left(1+\lambda R^\alpha\right) \|D_v^lf\|_{L_{t,x}^{p_\star}L_v^{q_\star}} \|g\|_{L_{t,x}^{p}L_v^{q}} + C\|f\|_{L_{t,x}^2H_v^k}^2&. 
\end{align*}
\end{enumerate}
Here $C>0$ is a constant depending only on $d,\beta,k,p,q$. 
\end{theorem}

\begin{theorem}\label{reg-p} 
Let the constants $\lambda=0$, $\alpha\in[0,1]$, $\beta\ge0$, $p,q\in[1,\infty]$, $T>0$. Let $f,g\in L_{t,x}^{p}L_v^{q}(\mathbb{R}^{1+2d})$ satisfy \eqref{KE} with $u=0$, and ${\rm supp\;\!} f,\,{\rm supp\;\!} g\subset[-T,T]\times\mathbb{R}^d\times\mathbb{R}^d$. If $D_v^kf\in L_{t,x}^{p}L_v^{q}(\mathbb{R}^{1+2d})$ for some $k>0$, then for any $p'\ge p$, $q'\ge q$ and $s\ge0$ such that 
\begin{align*}
0\le s<\frac{\left(1-\alpha\right)k}{1+\beta+k} + \left(\frac{1}{p'}-\frac{1}{p}\right)\left(d+\frac{\alpha+\beta+k}{1+\beta+k}\right) + \left(\frac{1}{q'}-\frac{1}{q}\right)\frac{d\left(1-\alpha\right)}{1+\beta+k}, 
\end{align*}
we have $D_x^sf\in L_{t,x}^{p'}L_v^{q'}(\mathbb{R}^{1+2d})$ with  
\begin{align*}
\|D_x^sf\|_{L_{t,x}^{p'}L_v^{q'}}
\le C_T\|D_v^kf\|_{L_{t,x}^{p}L_v^{q}} + C_T\|g\|_{L_{t,x}^{p}L_v^{q}}.  
\end{align*} 
Here $C_T>0$ is a constant depending only on $d,\alpha,\beta,k,p,q,p',q',T$. 
\end{theorem}

More generally, in the presence of inhomogeneous fluxes, our main result is as follows.  
\begin{theorem}\label{reg-u} 
Let the constants $\lambda=1$, $\alpha\in[0,1]$, $\beta\ge0$, $p,q\in(1,\infty)$, $R\ge1$, $s\in(0,1/2]$, $s_0\in[s,2s]$, $r\in[0,1/2)$, $r_0\in[0,2r]$, and $k,l\ge0$ such that 
\begin{align*}
r_0s+s_0r\ge2rs,\quad 
0\le l\le 1-2r+\beta,\quad 
0\le s\le \frac{(1-\alpha)(k+r)}{1+\beta+2k-l}. 
\end{align*}
Assume that the vector field $u\in L_{t,x}^{\infty}(\mathbb{R}^{1+d})$, and the functions $g\in L^{p}_{t,x}L_v^{q}(\mathbb{R}^{1+2d})$ and $f\in L_{t,x}^2H_v^k(\mathbb{R}^{1+2d})$ such that  $D_v^lf\in L_{t,x}^{p_\star}L_v^{q_\star}(\mathbb{R}^{1+2d})$ and  ${\rm supp\;\!} f\subset\mathbb{R}\times\mathbb{R}^d\times B_R$. If $f,g$ satisfy \eqref{KE} with $D_{t,x}^{s_0}u\in L_{t,x}^{a_1}(\mathbb{R}^{1+d})$ and $f\in L_{t,x}^{a_2}H_v^{1-2r+r_0}(\mathbb{R}^{1+2d})$ for the pair $(a_1,a_2)=(2,\infty)$ or $(\infty,2)$, then we have $f\in H_{t,x}^sH_v^{-r}(\mathbb{R}^{1+2d})$ with
\begin{align*}	
\left(1-2r\right)\|f\|_{H_{t,x}^sH_v^{-r}}^2 + \|f\|_{H_{t,x}^sH_v^{-1-r}}^2
\le CC_*^{\alpha+2s} \|D_v^lf\|_{L_{t,x}^{p_\star}L_v^{q_\star}} \|g\|_{L_{t,x}^{p}L_v^{q}}&\\ 
+ CC_*^{2s}\|f\|_{L_{t,x}^2H_v^k}^2 + CC_*^{4s}\left(1-2r\right)^{-1} \|D_{t,x}^{s_0}u\|_{L_{t,x}^{a_1}}^2\|f\|_{L_{t,x}^{a_2}H_v^{1-2r+r_0}}^2&. 
\end{align*}
Here the constant $C>0$ depends only on $d,\beta,k,p,q$, and $C_*:=R +\|u\|_{L_{t,x}^\infty}$. 
\end{theorem}

\subsection{Classical velocity averaging results}
The free-streaming transport operator $\partial_t + v\cdot\nabla_x$ is hyperbolic and propagates singularities at finite speeds. The averaging lemmas however ensure regularity for the moments of solutions to purely kinetic transport equations. 
\subsubsection{Classical averaging lemmas}
The classical theory typically states that for the functions $f,g\in L^{p}_{t,x,v}(\mathbb{R}^{1+2d})$ with $p\in(1,2]$ verifying \eqref{KE} with $u=0$, the moment $\rho_\phi[f]$ given by \eqref{af} satisfies 
\begin{align}\label{af-reg}
D_{t,x}^s\left(\rho_\phi[f]\right)\in L_{t,x}^{p}(\mathbb{R}^{1+d}),\quad s:=\frac{1-\alpha}{(1+\beta){p_\star}}. 
\end{align}
One may refer to \cite{GLPS,DL} for $p=2$ and to \cite{DPLM,Bez} for $p\in(1,2]$; see also \cite{DP,West,JV,TT} for more general cases. The classical averaging theory is based primarily on the ellipticity of the symbol of $\partial_t + v\cdot\nabla_x$ outside small neighborhoods of its characteristic directions, which provides regularity when averaged over the velocity variable. Furthermore, \eqref{af-reg} remains valid for general transport operator $\partial_t + a(v)\cdot\nabla_x$ with a smooth vector field $a:\mathbb{R}^d\rightarrow\mathbb{R}^d$ that is nondegenerate in the sense that for any compact subset $K\subset\mathbb{R}^d$, there is $C>0$ such that for any $\iota\in\mathbb{R}$, $\nu\in\mathbb{S}^{d-1}$, $\varepsilon>0$, 
\begin{align}\label{nonde}
\mathcal{L}^d(\left\{v\in K:\,\left|\;\!\iota+a(v)\cdot\nu\right|\le\varepsilon \right\})\le C\varepsilon, 
\end{align}
where $\mathcal{L}^d$ is the Lebesgue measure in $\mathbb{R}^d$; one may refer for instance to \cite[Section~1.5]{BGP}. This idea of exploiting ellipticity of transport operators relies essentially on arguments of characteristics. It provides information about the averaged quantity \eqref{af}, as contrasted with the precise negative Sobolev power in velocity indicated in \cite{JLT} and part~\ref{reg-1} of Theorem~\ref{reg-0}. More importantly, it restricts the possibility of accommodating the macroscopic velocity field $u$ as appeared in \eqref{KE}. 

\subsubsection{Optimality and possible improvements}\label{optimality}
The optimality of \eqref{af-reg} has been discussed in \cite{Ge,Lions1,Lions2}; see also \cite{DP,JP,JV}. Nevertheless, there is still room for improving averaging lemmas when additional information about solutions is known. It has been noticed in \cite{GG,JP,Bouchut,AM1} that assuming extra regularity for solutions in velocity gives better regularity than \eqref{af-reg}. This observation has led to improvements of the regularizing effect for assorted kinetic formulations; see \cite{Pe,JP}. Some sharpness of gain of the number of derivatives on this aspect was presented in \cite{DLW}. Moreover, incorporating stronger integrability properties of the solutions may lead to further improvements. By delving into the dual exponents of integrability between the functions $f$ and $g$ in \eqref{KE}, we can leverage the duality assumption to enhance the result to some extent. 

\subsubsection{Endpoint results}\label{l1endpoint}
Neither regularity like \eqref{af-reg} nor compactness can be expected for the moment $\rho_\phi[f]$ in the endpoint $L^1$ case due to concentrations in the velocity variable as indicated in \cite{GLPS}. A compactness property $\rho_\phi[f]$ was proven in \cite{GSR} to persist under an equi-integrability condition on the velocity variable that effectively rules out the pathological possibility. Relevant generalizations regarding the properties of the solution itself in the $L^1$ setting can be also derived from additional assumptions solely on the velocity variable; see \cite{ALSR,AM1}. 

\subsection{Hypoellipticity}
The phenomenon discussed in \S~\ref{optimality} and \S~\ref{l1endpoint} bears resemblance to the concept of hypoellipticity. The regularity in the time-space variables can be further improved when the solution $f$ to \eqref{KE} possesses additional smoothness in the velocity variable. The investigation of transferring regularity from the velocity variable to other variables for the solution itself has been explored in \cite{GG,Bouchut,ALSR}. 

\subsubsection{Hypoelliptic operators and commutator method}
Recall that the transport operator $\partial_t + v\cdot\nabla_x$ and the diffusion $-\Delta_v$ collectively constitute the so-called Kolmogorov operator $\partial_t + v\cdot\nabla_x-\Delta_v$. It serves as a prototypical hypoelliptic operator, initially examined in \cite{Kol} for its smoothing effects, as evidenced by an explicit calculation of its fundamental solution. A pioneering study of hypoelliptic second-order operators in \cite{Ho} reveals that the regularity property of $\partial_t + v\cdot\nabla_x-\Delta_v$ is related to the bracket condition 
\begin{align}\label{na0}
[\;\!\nabla_v,\;\!\partial_t + v\cdot\nabla_x\;\!]=\nabla_x. 
\end{align}
The arguments of \cite{Ho} are grounded in the characteristics of vector fields, specifically employing the method of fundamental solutions within the $L^2$ framework. Building upon \eqref{na0}, a direct $L^2$ energy-based commutator method was discovered in \cite{OR,Ko} to achieve the hypoelliptic regularity. A simplified version applicable to kinetic transport equations can be found in \cite[Section~3]{Bouchut}. 

\subsubsection{$L^p$ theory and endpoint case}\label{lpendpoint}
In contrast to the aforementioned works within the $L^2$ context, the hypoelliptic second-order counterpart of the Calder\'on-Zygmund theory, encompassing an $L^p$ framework ($1<p<\infty$) with more refined and optimal estimates, was developed in \cite{FS,Fo,RS} essentially through the method of fundamental solutions. For discussions on the analog of fractional order hypoellipticity, one may refer to \cite{Bouchut,HMP,ChenZhang}. It is well-established that the regularity estimates based on the singular integral theory do not extend to the endpoint $L^1$ case; see for instance \cite{Ornstein,KK}. Despite the failure of singular integral operators to be bounded in endpoint spaces, there remains a reasonable expectation for the gain of derivatives of order below the one given by the Calder\'on-Zygmund theory. 
 
\subsection{Averaging lemmas meet hypoellipticity}
Recent observations have shown that hypoelliptic methods can be applied to recover specific averaging lemmas. We aim in this work at shedding more complete light on the regularity properties for solutions to \eqref{KE}. Through the utilization of the commutator and energy methods employed in \cite{Bouchut,JLT,JLTs,AL}, we characterize the velocity averaging in part~\ref{reg-1} of Theorem~\ref{reg-0} and establish the hypoelliptic regularity for the solution itself in part~\ref{reg-2} of Theorem~\ref{reg-0}, which offers insight into the time-space regularization for \eqref{KE} with $u=0$. 

\subsubsection{General advection velocity}\label{inhomo-intro}
While most velocity averaging results bear on the streaming operator $v\cdot\nabla_x$, we examine the operator $f=f(t,x,v)\mapsto\nabla_x\cdot(vf+uf)$ in the presence of the advection velocity $u=u(t,x)$. Previous microlocal analysis techniques used in \cite{Gerard,Gerard91,GG} required the vector field $u$ to be smooth. The quantification of the smoothness for general advection velocity within a perturbative scenario was addressed in \cite{JLTs}. 

We leverage the characteristics-independent nature of the commutator identity~\eqref{na0}. This paves the way for exploiting the noncommutation property 
\begin{align}\label{na}
\begin{aligned}
[\;\!\nabla_v,\;\!\partial_t+\nabla_x\cdot((v+u)\cdot)\;\!]=\nabla_x
\end{aligned}
\end{align} 
for the analysis of \eqref{KE} in the presence of a general macroscopic velocity field $u=u(t,x)$. 
Our work culminates in the formulation of Theorem~\ref{reg-u}, demonstrating a gain in time-space regularity compensated by a corresponding loss of regularity in the velocity variable for solutions. Notably, the regularity required for the advection velocity~$u$ does not necessarily exceed that obtained for moments of the solution $f$ to \eqref{KE}. Note that \eqref{na} crucially relies on the flux $u$ being independent of the variable $v$. For a more general field $U(t,x,v)$ replacing $v+u(t,x)$, this relation deteriorates. In such cases, one can only expect a regularization effect if $U(t,x,v)$ is sufficiently smooth and nondegenerate in an appropriate sense; see, for instance, \cite{Gerard,Gerard91,GG} and \cite{JLTs}. 

\subsubsection{Endpoint regularity}\label{end-intro}
Inspired by the studies in \cite{PP,IS,GM} concerning the gain of integrability for solutions to hypoelliptic equations, we also provide an elementary proof of an endpoint variant of the hypoelliptic regularity in Theorem~\ref{reg-p}. The proof hinges on the use of fundamental solutions of general Kolmogorov equations, which captures the $L^1$ regularity theory of classical averaging lemmas and hypoellipticity (see \S~\ref{l1endpoint} and \S~\ref{lpendpoint}). 

We remark that a limitation of both approaches mentioned in \S\ref{inhomo-intro} and \S\ref{end-intro} lies in our inability to address the general transport operator $\partial_t+a(v)\cdot\nabla_x$ satisfying the nondegenerate condition~\eqref{nonde}. 

\subsection{Organization of the paper}
The outline of the article is as follows. Section~\ref{apps} presents various applications of our main results. In Section \ref{CE}, we detail the commutator and energy estimates, which incorporate the treatment of inhomogeneous fluxes and a regularity transfer mechanism. This section provides the proofs for Theorems \ref{reg-0} and \ref{reg-u}. In Section~\ref{endpoint}, we establish the endpoint hypoelliptic regularity stated in Theorem~\ref{reg-p}. This involves an analysis of fundamental solutions to fractional Kolmogorov equations, with further details provided in Appendix \ref{append}.

\subsection{Notations}
Let us end this section by introducing some basic notations that will be used throughout the rest of the article. We use $\mathcal{F}_{z}$ to denote the Fourier transforms with respect to the variable $z$. Define the Japanese bracket as $\langle\cdot\rangle:=(1+|\cdot|^2)^{1/2}$. We write $X\lesssim Y$ to mean that $X\le CY$ for some constant $C>0$ depending only on $d,\alpha,\beta,k,p,q,p',q',\sigma$, and write $X\lesssim_QY$ to mean that $X\le C_QY$ for some  $C_Q>0$ depending additionally on a specified quantity $Q$ in context. 

\medskip\noindent\textbf{Acknowledgement.}
We would like to thank Fran\c{c}ois Golse, Zexing Li, Cl\'{e}ment Mouhot and Luis Silvestre for helpful discussions. 

\section{Applications}\label{apps}
This section explores how our main theorems presented in \S~\ref{main-thm} apply to the regularity theory of kinetic formulations of hyperbolic conservation laws and general Fokker-Planck-Kolmogorov equations, providing context before delving into the proofs of our main results. 

We consider Burgers' equation with transport in \S~\ref{Burgers}. More general heterogeneous fluxes in scalar conservation laws and the study of related kinetic formulations can be found in \cite{Dalibard}. The one-dimensional Euler system of isentropic gas dynamics is discussed in \S~\ref{Euler}. Further historical context can be found in \cite{Pe,Dafermos}. Applications to advection-diffusion equations of the general Fokker-Planck-Kolmogorov type are addressed in \S~\ref{kolmogorov-app}. Fractional Sobolev regularity for such equations with the free-streaming transport operator $\partial_t+v\cdot\nabla_x$ has been extensively studied; see for instance \cite{Bouchut,IS}. Other models involving advection velocity fields that depend on the variables $t$ and $x$ arise, for example, in the study of Nordstr\"{o}m theory \cite{Nords} and active suspensions \cite{SS}. 

\subsection{Kinetic formulations}\label{kineticfor}
Various motivations and applications of Theorem~\ref{reg-u} arise in addressing regularity problems within kinetic formulations. In the kinetic approach pioneered  in \cite{LPT1,LPT2}, certain nonlinear conservation laws can be lifted to kinds of linear equations by introducing an extra variable. 

\subsubsection{Burgers' equation with transport}\label{Burgers}
Let us consider the solution $w=w(t,x)$ to Burgers' equation with a rough transport $u=u(t,x)\in L^\infty\cap H_{loc}^{s_0}(\mathbb{R}_+\times\mathbb{R})$ for some $s_0>0$, 
\begin{align}\label{burgers-trans}
	\partial_tw+\partial_x\left(w^2/2+uw\right)=0 {\quad\rm in\ }\mathbb{R}_+\times\mathbb{R}.  
\end{align}
The theory of kinetic formulations developed in \cite{LPT1} (see also \cite{Pe}) states that $w\in C(\mathbb{R}_+;L^1(\mathbb{R}))$ is an entropy solution to \eqref{burgers-trans} if and only if there is a nonnegative Radon measure $\mu=\mu(t,x,v)$ such that  
\begin{align}\label{burgers-kinetic}
\partial_tf+\partial_x\left(vf+uf\right)=\partial_v\mu {\quad\rm in\ }\mathbb{R}_+\times\mathbb{R}\times\mathbb{R}.  
\end{align}
Here $f=f(t,x,v)$ is the local Maxwellian defined as
\begin{align}\label{def-max}
f(t,x,v)=\begin{cases}1,&\text{ for }\ 0<v<u(t,x), \\ 
	-1, & \text{ for }\ u(t,x)<v<0, \\ 
	0, & \text { otherwise. }\end{cases}
\end{align}
It therefore satisfies $w=\int_\mathbb{R} f\dif v$. Furthermore, if the initial data $w|_{t=0}=w_0\in L^1\cap L^\infty(\mathbb{R})$, then $w(t,\cdot)\in L^1\cap L^\infty(\mathbb{R})$ and the total mass of the Radon measure $\mu$ is bounded. 

\begin{theorem}
Let $u\in L^\infty\cap H_{loc}^{s_0}(\mathbb{R}_+\times\mathbb{R})$ for some $s_0>0$. For any entropy solution $w\in C(\mathbb{R}_+;L^1(\mathbb{R}))$ to \eqref{burgers-trans} with $w|_{t=0}=w_0\in L^1\cap L^\infty(\mathbb{R})$, we have $w\in H_{loc}^s(\mathbb{R}_+\times\mathbb{R})$ for all $s\in(0,\min\{s_0,1/3\})$.  
\end{theorem}

\begin{proof}
Without loss of generality, we may assume that $s_0\in(0,1/3]$. We aim to show that $f\in H_{t,x}^{s_0-\varepsilon}H_v^{\varepsilon-1/2}$ for any $\varepsilon\in(0,s_0/2]$. To achieve this, we are going to apply Theorem~\ref{reg-u} to the equation \eqref{burgers-kinetic} with parameters $p=q\in(1,\infty)$, $\alpha=\varepsilon^2$, $\beta=1+\varepsilon^2$, $l=0$, $k=r=1/2-\varepsilon$, $r_0=1/2-5\varepsilon/2$ and $(a_1,a_2)=(2,\infty)$. First, since $1-2r+r_0=1/2-\varepsilon/2<1/2$, and $u\in L_{t,x}^\infty$, and $f$ is an indicator function in $v$ due to \eqref{def-max}, we know that locally $f\in L_{t,x,v}^\infty\cap L_{t,x}^\infty H_v^{1-2r+r_0}$. Next, regarding to the right-hand side of \eqref{burgers-kinetic}, the bounded Radon measure $\mu\in W_{t,x,v}^{-\alpha,p}$ for any $1<p<(1-\alpha/3)^{-1}$, since $W_{t,x,v}^{\alpha,{p_\star}}$ embeds into the space of continuous bounded functions for ${p_\star}>3/\alpha$. Now, we verify the parameter conditions of Theorem~\ref{reg-u} with $s=s_0-\varepsilon$: 
\begin{align*}
\frac{r_0}{r}+\frac{s_0}{s}=2-\frac{3\varepsilon}{1-2\varepsilon}+\frac{\varepsilon}{s_0-\varepsilon}\ge 2-\frac{3\varepsilon}{1-2\varepsilon}+\frac{\varepsilon}{1/3-\varepsilon}\ge2, 
\end{align*}
that is, $r_0s+rs_0\ge2rs$; besides, by choosing $\alpha=\varepsilon^2$, $\beta=1+\varepsilon^2$, $l=0$, $k=r=1/2-\varepsilon$, 
\begin{align*}
\frac{(1-\alpha)(k+r)}{1+\beta+2k-l}=\frac{(1-\varepsilon^2)(1-2\varepsilon)}{3+\varepsilon^2+2\varepsilon}\ge\frac{(1-\varepsilon)(1-2\varepsilon)}{3}\ge \frac{1}{3}-\varepsilon. 
\end{align*}
Applying Theorem~\ref{reg-u} to \eqref{burgers-kinetic} with the above specified parameters then yields that locally $D_{t,x}^sf\in L_{t,x}^2H_v^{\varepsilon-1/2}$. We hence obtain the local regularity for the averaged quantity $w=\int_\mathbb{R} f\dif v$. 
\end{proof}

\subsubsection{Isentropic gas dynamics}\label{Euler}
One may also observe conditional regularity results for solutions to the one-dimensional Euler system of isentropic gas dynamics, with its kinetic formulation introduced in \cite{LPT2}. Specifically, we are concerned with the following $2\times2$ hyperbolic system of conservation laws,
\begin{align}\label{isen}
\left\{\begin{aligned}
\ &\partial_t\rho+\partial_x(\rho u)=0, \\
\ &\partial_t(\rho u)+\partial_x(\rho u^2+p(\rho))=0, \\
\end{aligned}\right. 
\end{align} 
where the unknowns are the (nonnegative) density $\rho=\rho(t,x)$ and the velocity field $u=u(t,x)$ for $(t,x)\in\mathbb{R}_+\times\mathbb{R}$, and the pressure $p(\rho)=\frac{(\gamma-1)^2}{4\gamma}\rho^\gamma$ for $\gamma\in(1,3]$. According to \cite{LPT2,Pe}, the system \eqref{isen} admits the kinetic formulation with respect to its entropy solution $(\rho,u)$, 
\begin{align}\label{isen-kin}
\partial_t f+ \partial_x\left(\theta vf+(1-\theta)uf\right) =-\partial^2_v\mu.  
\end{align} 
Here $\mu=\mu(t,x,v)$ is a nonnegative bounded measure, and the local Maxwellian $f=f(t,x,v)$, parameterized by $\rho,u$ and associated with the constant $\theta\in(0,1]$, is given by 
\begin{align*} f=C_\theta\big[\rho^{2\theta}-(v-u)^2\big]_+^\frac{1-\theta}{2\theta} {\quad\rm for\ \ } \theta=\frac{\gamma-1}{2},  
\end{align*}
where $[\cdot]_+$ denotes the positive part, and $C_\theta>0$ is the normalizing constant ensuring that one can recover the density $\rho=\int_\mathbb{R} f\dif v$ and the momentum $\rho u=\int_\mathbb{R} vf\dif v$ in \eqref{isen}. In particular when $\theta=1$, we have $f=\frac{1}{2}\mathds{1}_{[u-\rho,u+\rho]}(v)$. 

Unfortunately, despite the insights provided by Theorem~\ref{reg-u}, establishing unconditional regularity for $f$ satisfying \eqref{isen-kin} remains challenging for general $\gamma\in(1,3)$. 

In the purely kinetic advection regime where $\gamma=3$ or equivalently, $\theta=1$, we can directly apply part~\ref{reg-1} of Theorem~\ref{reg-0} to \eqref{isen-kin} with parameters $p=q\approx1$, $\alpha\approx0$, $\beta\approx2$, $k\approx1/2$, and $l=0$ to obtain the local regularity $f\in H_{t,x}^sH_v^{-3/2}$ for any $s<1/4$, which thus implies that locally $\rho,\rho u\in H_{t,x}^s$ for all $s<1/4$.  It is a slight improvement compared to the results in \cite{LPT2,JP}. 

\subsection{General Fokker-Planck-Kolmogorov equations}\label{kolmogorov-app}
Let us consider a class of hypoelliptic equations whose solutions naturally exhibit regularity in the velocity variable as a consequence of energy estimates. The hypoelliptic regularization results from \S~\ref{main-thm} then provide the time-space regularity for the distribution function. 

Let $\sigma>0$. We are interested in the following advection-diffusion equation for an unknown distribution function $f=f(t,x,v)$ with $(t,x,v)\in\mathbb{R}\times\mathbb{R}^d\times\mathbb{R}^d$, 
\begin{align}\label{gkol}
\partial_tf+v\cdot\nabla_xf+u\cdot\nabla_xf+(-\Delta_v)^\sigma f=S, 
\end{align} 
where $u=u(t.x):\mathbb{R}\times\mathbb{R}^d\rightarrow\mathbb{R}^d$ is a divergence-free vector field (i.e., $\nabla_x\cdot u=0$) and $S=S(t,x,v)$ is a source term. 

\begin{theorem}
Let $\sigma,s_0>0$, and let the function $f\in L_t^\infty L_{x,v}^2\cap L_{t,x}^2H_v^\sigma(\mathbb{R}^{1+2d})$. Suppose that $f$ satisfies $(\partial_t+v\cdot\nabla_x+u\cdot\nabla_x)f\in L_{t,x}^2H_v^{-\sigma}(\mathbb{R}^{1+2d})$ as well as \eqref{gkol} with $u,D_{t,x}^{s_0}u\in L^\infty(\mathbb{R}^{1+d})$ and $S\in L_{t,x}^2H_{v}^{-\sigma}(\mathbb{R}^{1+2d})$. Assume further that $f$ is localized so that ${\rm supp\;\!} f\subset[-R,R]\times\mathbb{R}^d\times B_R$ for some $R\ge1$. Then we have $f\in L_{v}^2H_{t,x}^s(\mathbb{R}^{1+2d})$ for any 
\begin{align*}
s\in\Big(0,\frac{\sigma}{1+2\sigma}\Big] {\quad with\quad}
\begin{cases} \,s<\frac{s_0\sigma}{(1-\sigma)(1+2\sigma)},&\text{ if }\ \, \sigma\in(0,\frac{1}{2}], \\ 
\, s\le s_0\min\{1,\sigma\}, & \text{ if }\ \, \sigma>\frac{1}{2}. 
\end{cases}
\end{align*}
More precisely, there is some $C_0>0$ depending only on $\sigma,s_0,s,R$ such that 
\begin{align*}
\|f\|_{L_{v}^2H_{t,x}^s}
\le C_0\big(1+\|u\|_{L_{t,x}^\infty}^{2s}\big)\big(1+\|D_{t,x}^{s_0}u\|_{L_{t,x}^\infty}\big) \|S\|_{L_{t,x}^2H_{v}^{-\sigma}}. 
\end{align*}
\end{theorem}

\begin{proof}
Integrating \eqref{gkol} against $f$ yields that  
\begin{align*}
\frac{1}{2}\int_{\mathbb{R}^{1+2d}}(\partial_t+v\cdot\nabla_x+u\cdot\nabla_x)f^2+\int_{\mathbb{R}^{1+2d}}|D_v^\sigma f|^2=\int_{\mathbb{R}^{1+2d}}Sf. 
\end{align*} 
Using the divergence-free condition for $u$ and integration by parts, we have 
\begin{align*}
\frac{1}{2}\|f\|_{L_t^\infty L_{x,v}^2}^2 + \|D_v^\sigma f\|_{L_{t,x,v}^2}^2 
\le 2\|S\|_{L_{t,x}^2H_{v}^{-\sigma}}\big(\|f\|_{L_{t,x,v}^2}+\|D_v^\sigma f\|_{L_{t,x,v}^2}\big). 
\end{align*} 
By the Cauchy-Schwarz inequality and the fact that $f$ is supported in the time interval $t\in[-R,R]$ with $R\ge1$, we derive the energy estimate 
\begin{align}\label{gkol-energy}
\|f\|_{L_{t,x,v}^2} + \|D_v^\sigma f\|_{L_{t,x,v}^2} \lesssim R\|S\|_{L_{t,x}^2H_{v}^{-\sigma}}. 
\end{align} 
Let us assume without loss of generality that $\sigma\in(0,1]$, and take $\delta\in(0,\sigma]$ to be determined. Given that $u$ is divergence-free, we apply Theorem~\ref{reg-u} to \eqref{gkol} with parameters $p=q=2$, $\alpha=0$, $\beta=k=l=\sigma$, $r=1/2-\delta/2$, $r_0=\sigma-\delta$ and $(a_1,a_2)=(\infty,2)$. To be more precise, we define 
\begin{align}\label{gkol-s}
s':=\frac{(1-\alpha)(k+r)}{1+\beta+2k-l}=\frac{2\sigma+1-\delta}{2+4\sigma}
{\quad\rm and\quad} C_*:=R +\|u\|_{L_{t,x}^\infty}. 
\end{align} 
If the condition $r_0s'+s_0r\ge2rs'$ is satisfied, which is equivalent to
\begin{align}\label{gkol-ss}
(1-\delta)s_0\ge2(1-\sigma)s', 
\end{align} 
then we have 
\begin{align}\label{gkol-regu}
\begin{aligned}
\sqrt{\delta}\|f\|_{H_{t,x}^{s'}H_v^{-r}}
\lesssim&\ C_*^{s'}\|S\|_{L_{t,x}^2H_{v}^{-\sigma}} + C_*^{s'} \|f\|_{L_{t,x}^2H_v^\sigma} \\
&\,+C_*^{2s'}\delta^{-\frac{1}{2}}\|D_{t,x}^{s_0}u\|_{L_{t,x}^\infty}\|f\|_{L_{t,x}^2H_v^\sigma}. 
\end{aligned}
\end{align}
Gathering \eqref{gkol-energy} and \eqref{gkol-regu}, we obtain
\begin{align*}
\|f\|_{H_{t,x}^{s'}H_v^{-r}}
\lesssim_{\delta,R} C_*^{s'} \big(1+C_*^{s'}\|D_{t,x}^{s_0}u\|_{L_{t,x}^\infty}\big) \|S\|_{L_{t,x}^2H_{v}^{-\sigma}}. 
\end{align*}
In view of \eqref{gkol-s}, we have the interpolation  
\begin{align*}	
\|f\|_{L_{v}^2H_{t,x}^s} \le \|f\|_{H_{t,x}^{s'}H_v^{-r}}^{\vartheta} \|f\|_{L_{t,x}^2H_v^\sigma}^{1-\vartheta} {\quad\rm for\ \ } \vartheta=\frac{\sigma}{r+\sigma},\ \ s\le\vartheta s'=\frac{\sigma}{1+2\sigma}. 
\end{align*}
Combining the above two estimates with \eqref{gkol-energy}, we deduce that for any $0\le s\le\vartheta s'$ with $s'$ satisfying \eqref{gkol-s} and \eqref{gkol-ss}, 
\begin{align*}
\|f\|_{L_{v}^2H_{t,x}^s}
\lesssim_{\delta,R} C_*^{s} \big(1 + C_*^{s} \|D_{t,x}^{s_0}u\|_{L_{t,x}^\infty}^\vartheta\big) \|S\|_{L_{t,x}^2H_{v}^{-\sigma}}. 
\end{align*}
Considering the condition~\eqref{gkol-ss}, when $\sigma>1/2$, we can choose $s_0=s'$ and $\delta=2\sigma-1\in(0,\sigma]$ so that 
\begin{align*}	
r=\frac{1}{2}-\frac{\delta}{2}=1-\sigma {\quad\rm and\quad} s\le\vartheta s'=\frac{s_0\sigma}{r+\sigma} = s_0\sigma.  
\end{align*}
For $\sigma\in(0,1/2]$, the condition~\eqref{gkol-ss} is equivalent to
\begin{align*}	
\vartheta s'\le \frac{\sigma}{r+\sigma} \frac{s_0(1-\delta)}{2(1-\sigma)} =\frac{s_0\sigma}{1-\sigma} \frac{1-\delta}{1-\delta+2\sigma}. 
\end{align*}
Therefore, for any $s<\frac{s_0\sigma}{(1-\sigma)(1+2\sigma)}$ such that $s\le\vartheta s'$, there is some $\delta\in(0,\sigma]$ satisfying the above condition. This concludes the proof. 
\end{proof}

Focusing on the second-order instance of \eqref{gkol} with $\sigma=1$, we have the following result. 
\begin{corollary}\label{kol-sobolev}
For the $d\times d$ real symmetric matrix $A=A(t,x,v)$ and the $d$-dimensional real vectors $B=B(t,x,v)$ and $u=u(t,x)$, we assume that there is some constant $\Lambda>1$ such that for any $(t,x,v)\in\mathbb{R}^{1+2d}$, 
\begin{align*}
\Lambda^{-1}|\xi|^2\le A\;\!\xi\cdot\xi\le\Lambda|\xi|^2{\quad\rm for\ any\ } \xi\in\mathbb{R}^d,\quad |B|+|u|\le\Lambda. 
\end{align*} 
Suppose that $u$ is divergence-free, and $u,D_{t,x}^{s}u\in L^\infty(\mathbb{R}^{1+d})$ for some $s\in(0,1/3]$, and $S\in L_{t,x,v}^2(\mathbb{R}^{1+2d})$. If the function $f\in L_t^\infty L_{x,v}^2\cap L_{t,x}^2H_v^1(\mathbb{R}^{1+2d})$ is localized so that ${\rm supp\;\!} f\subset[-R,R]\times\mathbb{R}^d\times B_R$ for some $R\ge1$, and satisfies $(\partial_t+v\cdot\nabla_x+u\cdot\nabla_x)f\in L_{t,x}^2H_v^{-1}(\mathbb{R}^{1+2d})$ as well as
\begin{align}\label{gkol2}
\partial_tf+v\cdot\nabla_xf+u\cdot\nabla_xf=\nabla_v\cdot(A\nabla_vf)+B\cdot\nabla_vf+S, 
\end{align} 
then there is some $C_\Lambda>0$ depending only on $s,R,\Lambda$ such that  
\begin{align*}
\|f\|_{L_{v}^2H_{t,x}^s}
\le C_\Lambda \big(1 + \|D_{t,x}^{s}u\|_{L_{t,x}^\infty}\big) \|S\|_{L_{t,x,v}^2}. 
\end{align*}
\end{corollary}
It is also reasonable to expect pointwise regularity if the advection velocity $u$ possesses the same amount of regularity. While relevant, this issue diverges from the primary focus of this article. 

\section{Commutator and energy estimates}\label{CE}
This section is devoted to the proof of Proposition~\ref{reg-est} (see below), which directly implies the results of Theorem~\ref{reg-0} and Theorem~\ref{reg-u}. We first set up commutator and energy estimates that account for inhomogeneous fluxes. We then develop a mechanism for transferring regularity at the microlocal level, which enables us to demonstrate how the aforementioned estimates contribute to time-space regularity properties for solutions to \eqref{KE}. It is noteworthy that the estimates presented in this section are of an $L^2$ nature. 

Let us begin by introducing some abbreviations for simplicity. Throughout this section, we use $\widehat{\cdot}$ to denote the Fourier transform with respect to all variables $(t,x,v)\mapsto(\tau,\xi,\eta)$, and adopt the shorthand notation $\int\cdot:=\int_{\mathbb{R}\times\mathbb{R}^d\times\mathbb{R}^d}\cdot\dif\tau\dif\xi\dif\eta$. 

In order to refine the characterization of \eqref{na}, we have to examine the noncommutation relation at the microlocal level. To this end, we apply the Fourier transform to \eqref{KE} in all variables to see that 
\begin{align}\label{FKE}
\big(i\tau-\xi\cdot\nabla_\eta+i\xi\cdot\tilde{u}*_{\tau,\xi}\big)\widehat{f}
=\left(\lambda|\tau|^\alpha+|\xi|^\alpha\right)|\eta|^\beta\widehat{g},  
\end{align}
where we abbreviated $\tilde{u}:=\mathcal{F}_{t,x}u$. Let us define the operator  
\begin{align*}
X:=i\tau-\xi\cdot\nabla_\eta+i\xi\cdot\tilde{u}*_{\tau,\xi}.  
\end{align*}

\subsection{Commutator method}
Drawing inspiration from \eqref{na}, we introduce a constant $r\in\mathbb{R}$ and consider the commutator identity 
\begin{align}\label{naf}
\begin{aligned}
[\;\!\langle\eta\rangle^{1-r},\;\!X\;\!]
&=[\;\!\langle\eta\rangle^{1-r},\;-\xi\cdot\nabla_\eta\;\!]\\
&=(1-r)\,\xi\cdot\eta\,\langle\eta\rangle^{-1-r}. 
\end{aligned}
\end{align}
Similarly to \eqref{na}, this the commutator identity hinges on the fact that the vector field $u$ is independent of $v$, ensuring that the convolution operation $*_{\tau,\xi}$ commutes with multiplication by functions of $\eta$.

The energy estimates associated with the operator $X$ are stated as follows. 

\begin{lemma}\label{reg-lemma}
For any function $f(t,x,v)\in C^\infty_c(\mathbb{R}^{1+2d})$, any nonnegative function $\chi(\tau,\xi,\eta)\in C_b^\infty(\mathbb{R}^{1+2d})$, and any constants $s,r\ge0$, we have 
\begin{align}\label{reg-commu}	
\begin{aligned}
\int \chi\, |\xi|^2\langle\xi\rangle^{2s-2}\langle\eta\rangle^{-2-2r}\left(1+(1-2r)|\eta|^2\right)|\widehat{f}|^2
\le \int|\nabla_\eta\chi|\langle\xi\rangle^{2s}\langle\eta\rangle^{1-2r}|\widehat{f}|^2&\\
+ 2\left| \int \chi\,\xi\cdot\eta\,\langle\xi\rangle^{2s-2}\langle\eta\rangle^{-2r}X\widehat{f}\, \overline{\widehat{f}}\,\right|
+ 2\Im \int \chi\,\xi\cdot\eta\,\langle\xi\rangle^{2s-2}\langle\eta\rangle^{-2r}\xi\cdot\tilde{u}*_{\tau,\xi}\widehat{f}\, \overline{\widehat{f}}&, 
\end{aligned}
\end{align}
and 
\begin{align}\label{reg-time} 
\begin{aligned}
\int \chi |\tau|^2\langle\tau\rangle^{2s-2}\langle\eta\rangle^{-2r} |\widehat{f}|^2 
= -\Re \int \chi\,\tau\langle\tau\rangle^{2s-2}\langle\eta\rangle^{-2r} \xi\cdot\widehat{vf}\,\overline{\widehat{f}}&\\
+\Im\int\chi\,\tau\langle\tau\rangle^{2s-2}\langle\eta\rangle^{-2r} X\widehat{f}\,\overline{\widehat{f}}
-\Re \int\chi\,\tau\langle\tau\rangle^{2s-2}\langle\eta\rangle^{-2r} \xi\cdot\tilde{u}*_{\tau,\xi}\widehat{f}\,\overline{\widehat{f}}&. 
\end{aligned}
\end{align}
\end{lemma}

\begin{proof}
Let us consider the commutator identity \eqref{naf} acting on $\widehat{f}\;\!$. By integrating it against $\chi\;\!\xi\cdot\eta\;\!\langle\xi\rangle^{2s-2}\langle\eta\rangle^{-1-r}\overline{\widehat{f}}$, we obtain 
\begin{align}\label{x-x-x}
\begin{aligned}
(1-r)\int \chi\,|\xi\cdot\eta|^2\langle\xi\rangle^{2s-2}\langle\eta\rangle^{-2-2r}|\widehat{f}|^2&\\
=\int \chi\,\xi\cdot\eta\,\langle\xi\rangle^{2s-2}\langle\eta\rangle^{-1-r}\widehat{f}\, \overline{\big(\langle\eta\rangle^{1-r}X\widehat{f}-X(\langle\eta\rangle^{1-r}\widehat{f})\big)}&. 
\end{aligned}
\end{align}
For the second term on the right-hand side, integration by parts yields 
\begin{align*}
-\int \chi\,\xi\cdot\eta\,\langle\xi\rangle^{2s-2}\langle\eta\rangle^{-1-r}\widehat{f}\, \overline{X(\langle\eta\rangle^{1-r}\widehat{f})}
=\int\chi\,\xi\cdot\eta\,\langle\xi\rangle^{2s-2}\langle\eta\rangle^{-1-r}X\widehat{f}\, \overline{\langle\eta\rangle^{1-r}\widehat{f}}&\\
-\int\langle\xi\rangle^{2s-2}\xi\cdot\nabla_\eta\left(\chi\,\xi\cdot\eta\,\langle\eta\rangle^{-1-r}\right) \widehat{f}\, \overline{\langle\eta\rangle^{1-r}\widehat{f}} &\\
+2\Im\int \chi\,\xi\cdot\eta\,\langle\xi\rangle^{2s-2}\langle\eta\rangle^{-1-r} \xi\cdot\tilde{u}*_{\tau,\xi}\widehat{f}\, \overline{\langle\eta\rangle^{1-r}\widehat{f}}&. 
\end{align*}
Note that the first term on the right-hand side above and the first term on the right-hand side of \eqref{x-x-x} can be combined as twice their real part. To evaluate the second term on the right-hand side above, we compute 
\begin{align*}
\xi\cdot\nabla_\eta\big(\chi\,\xi\cdot\eta\,\langle\eta\rangle^{-1-r}\big)&\\
= (\xi\cdot\nabla_\eta\chi)\,\xi\cdot\eta\,\langle\eta\rangle^{-1-r}
 + \chi\,\langle\eta\rangle^{-3-r} \big(|\xi|^2\langle\eta\rangle^2-(1+r)|\xi\cdot\eta|^2\big)&. 
\end{align*}
Combining the above three identities, we have 
\begin{align*}
\int \chi\,\langle\xi\rangle^{2s-2}\langle\eta\rangle^{-2-2r}\big(|\xi|^2\langle\eta\rangle^2-2r|\xi\cdot\eta|^2\big)|\widehat{f}|^2
= 2\Re \int \chi\,\xi\cdot\eta\,\langle\xi\rangle^{2s-2}\langle\eta\rangle^{-2r}X\widehat{f}\, \overline{\widehat{f}}&\\
- \int (\xi\cdot\nabla_\eta\chi)\,\xi\cdot\eta\,\langle\xi\rangle^{2s-2}\langle\eta\rangle^{-2r}|\widehat{f}|^2
+ 2\Im\int \chi\,\xi\cdot\eta\,\langle\xi\rangle^{2s-2}\langle\eta\rangle^{-2r} \xi\cdot\tilde{u}*_{\tau,\xi}\widehat{f}\, \overline{\widehat{f}}&.  
\end{align*}
This implies \eqref{reg-commu} by noticing that 
\begin{align*}
|\xi|^2\langle\eta\rangle^2-2r|\xi\cdot\eta|^2\ge|\xi|^2+(1-2r)|\xi|^2|\eta|^2. 
\end{align*}  

Besides, we observe that the operator $X$ satisfies  
\begin{align*}
|\tau|^2\widehat{f}= -\tau\xi\cdot\widehat{vf} -i\tau X\widehat{f} -\tau\xi\cdot\tilde{u}*_{\tau,\xi}\widehat{f}\,. 
\end{align*}
Integrating it against $\chi\;\!\langle\tau\rangle^{2s-2}\langle\eta\rangle^{-2r}\overline{\widehat{f}}$  gives \eqref{reg-time}. This completes the proof. 
\end{proof}

\subsection{Inhomogeneous fluxes}
Let us now focus on deriving the following trilinear-type estimates associated with the macroscopic advection velocity $u=u(t,x)$ in \eqref{KE}, which are integral to the energy estimates in Lemma~\ref{reg-lemma}. In comparison to the perturbative setting treated in \cite{JLTs}, achieving the general result without smallness and high-order smoothness necessitates additional velocity regularity for solutions to \eqref{KE}. 
\begin{lemma}\label{reg-tri}
Let the constants $s\in(0,1/2]$, $s_0\in(0,2s]$, $r\in[0,1/2]$, $r_0\in[0,2r]$, and $(a_1,a_2)=(2,\infty)$ or $(\infty,2)$. Assume that $\chi(\tau,\xi,\eta)\in C^\infty(\mathbb{R}^{1+2d})$ valued in $[0,1]$ satisfies $|\nabla_{(\tau,\xi)}\chi|\le C'\langle(\tau,\xi)\rangle^{-1}$ for some $C'\ge1$. If $u\in L_{t,x}^\infty$ and $D_{t,x}^{s_0}u\in L_{t,x}^{a_1}$, then for any function $f(t,x,v)\in C^\infty_c(\mathbb{R}^{1+2d})$, we have 
\begin{align}\label{reg-trix}
\begin{aligned}
\left|\,\Im \int \chi\,\xi\cdot\eta\,\langle\xi\rangle^{2s-2}\langle\eta\rangle^{-2r}\xi\cdot\tilde{u}*_{\tau,\xi}\widehat{f}\, \overline{\widehat{f}}\,\right|&\\
\lesssim C' \|D_{t,x}^{s_0}u\|_{L_{t,x}^{a_1}} \|D_{t,x}^{2s-s_0}f\|_{L_{t,x}^2H_v^{-r_0}} \|D_v^{1-2r+r_0}f\|_{L_{t,x}^{a_2}L_v^{2}}&; 
\end{aligned}
\end{align} 
and for any $r'\in\mathbb{R}$, we have 
\begin{align}\label{reg-trit}
\begin{aligned}
\int\chi|\tau|^2\langle\tau\rangle^{2s-2} \langle\eta\rangle^{-2r'}|\tilde{u}*_{\tau,\xi}\widehat{f}||\widehat{f}|
\lesssim \|u\|_{L_{t,x}^{a_1}} \|D_t^sf\|_{L^2_{t,x}H_v^{-r'}}^2 &\\
+ \|D_t^su\|_{L_{t,x}^{a_1}}\|f\|_{L_{t,x}^{a_2}H_v^{-r'}} \|D_t^sf\|_{L^2_{t,x}H^{-r'}_v}&. 
\end{aligned}
\end{align}
\end{lemma}

\begin{proof}
Let us first abbreviate 
\begin{align*}
\Phi(\varrho,\eta):=\chi(\tau,\xi,\eta)\,\xi\cdot\eta\,\langle\xi\rangle^{2s-2}\langle\eta\rangle^{-2r} \xi, 
\quad\varrho:=(\tau,\xi)\in\mathbb{R}^{1+d}. 
\end{align*}
In the light of \cite[Lemma~1]{JLTs}, we are able to write  
\begin{align}\label{exchange}
\begin{aligned}
2\Im \int \chi\,\xi\cdot\eta\,\langle\xi\rangle^{2s-2}\langle\eta\rangle^{-2r}\xi\cdot\tilde{u}*_{\tau,\xi}\widehat{f}\, \overline{\widehat{f}}
=2\Im\int_{\varrho,\eta} \Phi\cdot\tilde{u}*_\varrho\widehat{f}\, \overline{\widehat{f}}&\\
=\Im\int_{\varrho,\varrho',\eta} (\Phi(\varrho+\varrho',\eta)-\Phi(\varrho',\eta))\, \cdot\tilde{u}(\varrho)\widehat{f}(\varrho',\eta) \overline{\widehat{f}}(\varrho+\varrho',\eta)&. 
\end{aligned}
\end{align}
Indeed, by using a change of variables and taking the conjugate, we have 
\begin{align*}
\Im\int_{\varrho,\eta}\Phi\cdot\tilde{u}*_\varrho\widehat{f}\,\overline{\widehat{f}}
&=\Im\int_{\varrho,\varrho',\eta}\Phi(\varrho,\eta)\cdot\tilde{u}(\varrho-\varrho')\widehat{f}(\varrho',\eta)\,\overline{\widehat{f}}(\varrho,\eta)\\
&= \Im\int_{\varrho,\varrho',\eta} \Phi(\varrho',\eta)\cdot\tilde{u}(\varrho'-\varrho)\widehat{f}(\varrho,\eta)\,\overline{\widehat{f}}(\varrho',\eta)\\
&= -\Im\int_{\varrho,\varrho',\eta} \Phi(\varrho',\eta)\cdot\overline{\tilde{u}}(\varrho'-\varrho)\overline{\widehat{f}}(\varrho,\eta)\, \widehat{f}(\varrho',\eta)\\
&= -\Im\int_{\varrho,\varrho',\eta} \Phi(\varrho',\eta)\cdot\tilde{u}(\varrho-\varrho') \widehat{f}(\varrho',\eta)\,\overline{\widehat{f}}(\varrho,\eta), 
\end{align*}
where the variables $\rho$ and $\rho'$ were exchanged in the second line and conjugation was taken in the third. Changing variables $(\rho,\rho')\mapsto(\rho+\rho',\rho')$ then gives that 
\begin{align*}
\Im\int_{\varrho,\eta}\Phi\cdot\tilde{u}*_\varrho\widehat{f}\,\overline{\widehat{f}}
= -\Im\int_{\varrho,\varrho',\eta} \Phi(\varrho',\eta)\cdot\tilde{u}(\varrho) \widehat{f}(\varrho',\eta)\,\overline{\widehat{f}}(\varrho+\varrho',\eta)&. 
\end{align*}
Combining this with the following identity that comes from a change of variables
\begin{align*}
\Im\int_{\varrho,\eta} \Phi\cdot\tilde{u}*_\varrho\widehat{f}\,\overline{\widehat{f}}
&=\Im\int_{\varrho,\varrho',\eta}\Phi(\varrho,\eta)\cdot\tilde{u}(\varrho-\varrho')\widehat{f}(\varrho',\eta)\,\overline{\widehat{f}}(\varrho,\eta)&\\
&=\Im\int_{\varrho,\varrho',\eta} \Phi(\varrho+\varrho',\eta)\cdot\tilde{u}(\varrho)\widehat{f}(\varrho',\eta)\,\overline{\widehat{f}}(\varrho+\varrho',\eta), 
\end{align*}
we obtain \eqref{exchange} as claimed. 

Let us now consider the function $\psi:\mathbb{R}^{1+d}\times\mathbb{R}\rightarrow\mathbb{R}^d$ defined by  
\begin{align*}
\psi(\varrho,\eta):=\Phi(\varrho,\eta)\,|\varrho|^{-2s}|\eta|^{-1+2r}. 
\end{align*}
It turns out from the definition that $\psi(\varrho,\eta)$ is valued in $[-1,1]$. By the definitions of $\psi$ and $\Phi$, together with the assumption on $|\nabla_\varrho\chi|$, we have 
\begin{align}\label{grad-phi}
|\nabla_\varrho\psi(\varrho,\eta)|\lesssim C'\langle\varrho\rangle^{-1}.
\end{align}
Abbreviate $s_1:=2s-s_0\ge0$. For $\varrho,\varrho'\in\mathbb{R}^{1+d}$, $\eta\in\mathbb{R}^d$, we write 
\begin{align}\label{psi0}
\begin{aligned}
|\eta|^{-1+2r} \big(\Phi(\varrho+\varrho',\eta)-\Phi(\varrho',\eta)\big)&\\
= \Psi_1(\varrho,\varrho',\eta)\, |\varrho'|^{s_1}|\varrho|^{s_0}+ \Psi_2(\varrho,\varrho',\eta)\, |\varrho'|^{s_1} |\varrho|^{s_0} + \Psi_2(\varrho,\varrho',\eta)\, |\varrho+\varrho'|^{s_1} |\varrho|^{s_0}&,  
\end{aligned}
\end{align}
where $\Psi_1,\Psi_2$ are defined by 
\begin{align*}
\Psi_1(\varrho,\varrho',\eta)&:=(\psi(\varrho+\varrho',\eta)-\psi(\varrho',\eta))|\varrho'|^{2s}|\varrho|^{-s_0},\\
\Psi_2(\varrho,\varrho',\eta)&:=\psi(\varrho+\varrho',\eta) \big(|\varrho+\varrho'|^{2s}-|\varrho'|^{2s}\big) \big(|\varrho+\varrho'|^{s_1}+|\varrho'|^{s_1}\big)^{\!-1} |\varrho|^{-s_0}. 
\end{align*}
For $\Psi_1$, we see from \eqref{grad-phi} that if $2|\varrho|\le|\varrho'|$, then 
\begin{align*}
|\Psi_1(\varrho,\varrho',\eta)| &=\big|\psi(\varrho+\varrho',\eta)-\psi(\varrho',\eta)\big||\varrho'|^{2s}|\varrho|^{-s_0}\\
&\lesssim C'\langle\varrho'\rangle^{-1} |\varrho'|^{s_0}|\varrho|^{1-s_0} \le C'. 
\end{align*}
It thus implies that  
\begin{align}\label{psi1}
\sup\nolimits_{\varrho,\varrho',\eta}|\Psi_1|
\le \sup\nolimits_{\{2|\varrho|\ge|\varrho'|\}}|\Psi_1| + \sup\nolimits_{\{2|\varrho|\le|\varrho'|\}}|\Psi_1|
\lesssim C'&. 
\end{align}
For $\Psi_2$, by applying twice the elementary inequality  $\big||\varrho_1+\varrho_2|^{a}-|\varrho_2|^{a}\big|\le|\varrho_1|^{a}$ for any $a\in[0,1]$ and $\varrho_1,\varrho_2\in\mathbb{R}^{1+d}$, we have  
\begin{align*}
\big||\varrho+\varrho'|^{2s}-|\varrho'|^{2s}\big| \big(|\varrho+\varrho'|^{s_1}+|\varrho'|^{s_1}\big)^{\!-1}|\varrho|^{-s_0}
\le |\varrho|^{2s-s_1-s_0}=1
\end{align*}
which implies that  
\begin{align}\label{psi2}
\sup\nolimits_{\varrho,\varrho',\eta} |\Psi_2| \le 1. 
\end{align}
In view of \eqref{exchange} and \eqref{psi0}, we know that 
\begin{align*}
2\Im \int \chi\,\xi\cdot\eta\,\langle\xi\rangle^{2s-2}\langle\eta\rangle^{-2r}\xi\cdot\tilde{u}*_{\tau,\xi}\widehat{f}\, \overline{\widehat{f}}\,
=2\Im\int_{\varrho,\eta} \Phi\cdot\tilde{u}*_\varrho\widehat{f}\, \overline{\widehat{f}}\,&\\
=\Im\int_{t,x,\eta,\varrho,\varrho'} (\Psi_1+\Psi_2) \cdot \mathcal{F}_{t,x}D_{t,x}^{s_0}u(\varrho) \widehat{D_{t,x}^{s_1}f}(\varrho',\eta) \mathcal{F}_vf(t,x,\eta) e^{i(t,x)\cdot(\varrho+\varrho')} |\eta|^{1-2r}&\\ 
+\Im\int_{t,x,\eta,\varrho,\varrho'} \Psi_2 \cdot \mathcal{F}_{t,x}D_{t,x}^{s_0}u(\varrho)\widehat{f}(\varrho',\eta) \mathcal{F}_vD_{t,x}^{s_1}f(t,x,\eta) e^{i(t,x)\cdot(\varrho+\varrho')} |\eta|^{1-2r}&. 
\end{align*}
With the aid of \eqref{psi1}, \eqref{psi2} and H\"older's inequality, we then arrive at \eqref{reg-trix}. 

Finally, \eqref{reg-trit} is a direct consequences of H\"older's inequality and the fractional Leibniz rule (see for instance \cite[Section~2.1]{Taylor}). 
\end{proof}

\subsection{Transfer of regularity}\label{transfer}
We are now positioned to establish the following regularity result by combining the estimates from preceding subsections through a transfer mechanism of the hypoelliptic nature. 
\begin{proposition}\label{reg-est} 
Let $(a_1,a_2)=(2,\infty)$ or $(\infty,2)$, $p,q\in(1,\infty)$, $s\in(0,1/2]$, $s_0\in(0,2s]$, $r\in[0,1/2]$, $r_0\in[0,2r]$, and $k,l\ge0$ such that 
\begin{align}\label{ranges}
0\le l\le 1-2r+\beta,\quad 0\le s\le \frac{(1-\alpha)(k+r)}{1+\beta+2k-l}.
\end{align}
Assume that the functions $g\in L^{p}_{t,x}L_v^{q}$ and $f\in L_{t,x}^2H_v^k$ such that $D_v^lf\in L_{t,x}^{p_\star}L_v^{q_\star}$, $f\in L_{t,x}^{a_2}H_v^{1-2r+r_0}$ and $\sup\nolimits_{{\rm supp\;\!} f}|v|<\infty$. If $f,g$ satisfy \eqref{KE} with $u\in L_{t,x}^\infty$ and $D_{t,x}^{s_0}u\in L_{t,x}^{a_1}$, then we have $f\in H_{t,x}^sH_v^{-r}$ with 
\begin{align*}
\begin{aligned}
\left(1-2r\right)\|(D_x^s,c_*^{s}D_t^s)f\|_{L_{t,x}^2H_v^{-r}}^2
 + \|(D_x^s,c_*^{s}D_t^s)f\|_{L_{t,x}^2H_v^{-1-r}}^2\\
\lesssim \left(1+\lambda c_*^{-\alpha}\right) \|D_v^lf\|_{L_{t,x}^{p_\star}L_v^{q_\star}} \|g\|_{L_{t,x}^{p}L_v^{q}} 
+\|f\|_{L_{t,x}^2H_v^k}^2  \\
 +\|D_{t,x}^{s_0}u\|_{L_{t,x}^{a_1}} \|f\|_{L_{t,x}^{a_2}H_v^{1-2r+r_0}} \big(\|D_t^sf\|_{L^2_{t,x}H^{-r}_v} +\|D_{t,x}^{2s-s_0}f\|_{L_{t,x}^2H_v^{-r_0}}\big)&, 
\end{aligned}
\end{align*}
where $c_*:=(\sup\nolimits_{{\rm supp\;\!} f}\langle v\rangle+\|u\|_{L_{t,x}^\infty})^{-1}$. 
\end{proposition}

\begin{proof}
In order to address the issue of regularity transfer, we consider $\delta,c_1\in(0,1]$, which will be determined later, and introduce three auxiliary functions defined in the Fourier space, $\chi_1=\chi_1(\tau,\xi,\eta)\in C^\infty(\mathbb{R}\times\mathbb{R}^d\times\mathbb{R}^d)$ and $\chi_2=\chi_2(\tau,\xi)$, $\chi_3=\chi_3(\tau,\xi)\in C^\infty(\mathbb{R}\times\mathbb{R}^d)$, all valued in $[0,1]$ and satisfying 
\begin{align*}
\chi_1=1 {\ \ \rm in\ } \{|(c_1\tau,\xi)|^\delta\ge|\eta|\}, \quad
\chi_1=0 {\ \ \rm in\ } \{2\langle(c_1\tau,\xi)\rangle^\delta\le|\eta|\},\\
\chi_2=1 {\ \ \rm in\ } \{c_1|\tau|\le1+|\xi|\}, \quad 
\chi_2=0 {\quad\rm in\ } \{c_1|\tau|\ge2+2|\xi|\} ,\\
\chi_3=1 {\ \ \rm in\ } \{c_1|\tau|\ge1+|\xi|\}, \quad 
\chi_3=0 {\quad\rm in\ } \{c_1|\tau|\le|\xi|\},\\
|\nabla_\eta\chi_1|\le 4\langle\eta\rangle^{-1},\quad
|\nabla_{(\tau,\xi)}(\chi_1\chi_2)|\lesssim \langle(\tau,\xi)\rangle^{-1} {\ \ \rm in\ }\mathbb{R}^{1+2d}. 
\end{align*}
We also suppose that for any $j\in\mathbb{N}^{1+2d}$ with $|j|\le1+d$, 
\begin{align}\label{Mi-cond}
|\partial^j(\chi_1\chi_2)|+|\partial^j(\chi_1\chi_3)|\lesssim1 {\ \ \rm in\ }\mathbb{R}^{1+2d}. 
\end{align}
The proof will then proceed in four steps. 

\medskip\noindent\textit{Step 1. Regularity outside the effective region of $\chi_1$. }\\
We consider the constant $\delta\in (0,1]$ such that 
\begin{align}\label{exponent1}
s/\delta-r\le k. 
\end{align}
It then follows that 
\begin{align}\label{trans-v}
\int_{\{|(c_1\tau,\xi)|^{\delta}\le|\eta|\}} \langle(c_1\tau,\xi)\rangle^{2s}\langle\eta\rangle^{-2r}|\widehat{f}|^2
\le \int \langle\eta\rangle^{2s/\delta-2r}|\widehat{f}|^2
= \int \langle\eta\rangle^{2k}|\widehat{f}|^2.  
\end{align}

\medskip\noindent\textit{Step 2. Regularity on the support of $\chi_1\chi_2$. }\\
By \eqref{reg-commu} of Lemma~\ref{reg-lemma} with $\chi:=\chi_1\chi_2$, we obtain 
\begin{align}\label{hand1}
\begin{aligned}
\int_{\{|\eta|\le|(c_1\tau,\xi)|^{\delta},\,c_1|\tau|\le1+|\xi|\}}  |\xi|^2\langle\xi\rangle^{2s-2}\langle\eta\rangle^{-2-2r}\left(1+(1-2r)|\eta|^2\right)|\widehat{f}|^2&\\
\le 4\int_{\{|\eta|\ge|(c_1\tau,\xi)|^\delta\}} \langle\eta\rangle^{-2r}\langle\xi\rangle^{2s}|\hat{f}|^2 
+2\left|\int \chi_1\chi_2\,\xi\cdot\eta\,\langle\xi\rangle^{2s-2}\langle\eta\rangle^{-2r}X\widehat{f}\, \overline{\widehat{f}}\,\right|&\\
+ 2\Im \int \chi_1\chi_2\,\xi\cdot\eta\,\langle\xi\rangle^{2s-2}\langle\eta\rangle^{-2r}\xi\cdot\tilde{u}*_{\tau,\xi}\widehat{f}\, \overline{\widehat{f}}\,&. 
\end{aligned}
\end{align}
In view of \eqref{FKE}, for the constant $l\ge0$, we write 
\begin{align}\label{Mikh}
\begin{aligned}
\int \chi_1\chi_2\,\xi\cdot\eta\,\langle\xi\rangle^{2s-2}\langle\eta\rangle^{-2r}X\widehat{f}\, \overline{\widehat{f}}&\\
= \int \chi_1\chi_2\,\xi\cdot\eta\,\langle\xi\rangle^{2s-2}\langle\eta\rangle^{-2r} \left(\lambda|\tau|^\alpha+|\xi|^\alpha\right)|\eta|^{\beta-l}\widehat{g}\, \overline{|\eta|^l\widehat{f}}&. 
\end{aligned}
\end{align}
Let us choose $s,\delta\ge0$ such that 
\begin{align}\label{exponent2}
2s-1+\alpha+\delta(1-2r+\beta-l)=0. 
\end{align}
Now that $c_1|\tau|\le2(1+|\xi|)$ and $|\eta|\le2\langle(c_1\tau,\xi)\rangle^\delta\lesssim \langle\xi\rangle^\delta$ hold on the support of $\chi_1\chi_2$, together with the smoothness assumption~\eqref{Mi-cond} on $\chi_1\chi_2$ and the fact that $1-2r+\beta-l\ge0$, we are able to apply the Mikhlin type multiplier theorem (one may refer for instance to \cite[Section~6.2]{Gra}; see also \cite{NI} for the case of mixed Lebesgue spaces) to \eqref{Mikh} with the Fourier multiplier
\begin{align*}
\chi_1\chi_2\,\xi\cdot\eta\,\langle\xi\rangle^{2s-2}\langle\eta\rangle^{-2r} \left(\lambda|\tau|^\alpha+|\xi|^\alpha\right)|\eta|^{\beta-l}.
\end{align*}
This application then allows us to see that 
\begin{align}\label{hand2}
\left|\int \chi_1\chi_2\,\xi\cdot\eta\,\langle\xi\rangle^{2s-2}\langle\eta\rangle^{-2r}X\widehat{f}\,\overline{\widehat{f}}\,\right| 
\lesssim \big(1+\lambda c_1^{-\alpha}\big) \|D_v^lf\|_{L_{t,x}^{p_\star}L_v^{q_\star}} \|g\|_{L_{t,x}^{p}L_v^{q}}. 
\end{align}
We thus conclude from \eqref{trans-v}, \eqref{hand1}, \eqref{hand2}, as well as \eqref{reg-trix} of Lemma~\ref{reg-tri}, that 
\begin{align}\label{trans-x}
\begin{aligned}
\int_{\{|\eta|\le|(c_1\tau,\xi)|^{\delta},\,c_1|\tau|\le1+|\xi|\}}  \langle(c_1\tau,\xi)\rangle^{2s}\langle\eta\rangle^{-2-2r} \left(1+(1-2r)|\eta|^2\right) |\widehat{f}|^2 &\\
\lesssim \|f\|_{L_{t,x}^2H_v^k}^2  
+ \big(1+\lambda c_1^{-\alpha}\big)\|D_v^lf\|_{L_{t,x}^{p_\star}L_v^{q_\star}} \|g\|_{L_{t,x}^{p}L_v^{q}}&\\
+\|D_{t,x}^{s_0}u\|_{L_{t,x}^{a_1}} \|D_{t,x}^{2s-s_0}f\|_{L_{t,x}^2H_v^{-r_0}} \|D_v^{1-2r+r_0}f\|_{L_{t,x}^{a_2}L_v^{2}}&.  
\end{aligned}
\end{align}
We point out that \eqref{exponent1} and \eqref{exponent2}, along with the assumption that $1-2r+\beta-l\ge0$, determine the range of $s$ given in \eqref{ranges}. 

\medskip\noindent\textit{Step 3. Regularity on the support of $\chi_1\chi_3$. }\\
Let us take the constant $r'=r$ or $1+r$ for brevity. 
Applying \eqref{reg-time} of Lemma~\ref{reg-lemma} with $\chi:=\chi_1\chi_3$ yields that 
\begin{align*}
 \int \chi_1\chi_3\,|\tau|^2\langle\tau\rangle^{2s-2}\langle\eta\rangle^{-2r'}|\widehat{f}|^2 
\le c_1\int_{\tau,\xi} |\tau|^2\langle\tau\rangle^{2s-2}\left|\int_\eta\langle\eta\rangle^{-2r'}\widehat{vf}\,\overline{\widehat{f}}\,\right|&\\
+ \left|\int\chi_1\chi_3\,\tau\langle\tau\rangle^{2s-2}\langle\eta\rangle^{-2r'} X\widehat{f}\,\overline{\widehat{f}}\,\right| + \int|\tau|^2\langle\tau\rangle^{2s-2}\langle\eta\rangle^{-2r'} |\tilde{u}*_{\tau,\xi}\widehat{f}||\widehat{f}|&. 
\end{align*}
We first notice that 
\begin{align*}
\left|\int_\eta\langle\eta\rangle^{-2r'}\widehat{vf}\,\overline{\widehat{f}}\,\right| =\|v\mathcal{F}_{t,x}f\|_{H_v^{-r'}}\|\mathcal{F}_{t,x}f\|_{H_v^{-r'}}
\lesssim\sup\nolimits_{{\rm supp\;\!} f}\langle v\rangle \int_\eta\langle\eta\rangle^{-2r'}|\widehat{f}|^2. 
\end{align*}
Next, similar calculations to those in \eqref{Mikh}, \eqref{exponent2}, \eqref{hand2} show that by the Mikhlin type multiplier theorem, 
\begin{align*}
\chi_1\chi_3\,\tau\langle\tau\rangle^{2s-2}\langle\eta\rangle^{-2r'}\left(\lambda|\tau|^\alpha+|\xi|^\alpha\right)|\eta|^{\beta-l}
\end{align*}
is a bounded Fourier multiplier on $L_{t,x}^pL_v^q(\mathbb{R}^{1+2d})$ so that from \eqref{FKE}, we have 
\begin{align*}
\left|\int\chi_1\chi_3\,\tau\langle\tau\rangle^{2s-2}\langle\eta\rangle^{-2r'} X\widehat{f}\,\overline{\widehat{f}}\,\right|
\lesssim \|D_v^lf\|_{L_{t,x}^{p_\star}L_v^{q_\star}} \|g\|_{L_{t,x}^{p}L_v^{q}}.  
\end{align*}
Taking $c_1:=c_2\min\big\{(\sup\nolimits_{{\rm supp\;\!}f}\langle v\rangle)^{-1},\|\langle u\rangle\|_{L_{t,x}^\infty}^{-1}\big\}$ for some $c_2\in(0,1]$ to be chosen, and gathering the above three estimates with \eqref{reg-trit} of Lemma~\ref{reg-tri}, we obtain 
\begin{align*}
\int \chi_1\chi_3\,|\tau|^{2s}\langle\eta\rangle^{-2r'} |\widehat{f}|^2
\lesssim c_2 \int |\tau|^{2s}\langle\eta\rangle^{-2r'}|\widehat{f}|^2 &\\
+ \|D_v^lf\|_{L_{t,x}^{p_\star}L_v^{q_\star}} \|g\|_{L_{t,x}^{p}L_v^{q}} 
+ \|D_t^su\|_{L^{a_1}_{t,x}}\|f\|_{L^{a_2}_{t,x}H_v^{-r'}} \|D_t^sf\|_{L^2_{t,x}H^{-r'}_v}&.
\end{align*}
By taking our choice of $\chi_1\chi_3$ into account, we deduce that 
\begin{align}\label{trans-t}
\begin{aligned}
\int_{\{|\eta|\le|(c_1\tau,\xi)|^\delta,\, c_1|\tau|\ge1+|\xi|\}} \langle(c_1\tau,\xi)\rangle^{2s} \langle\eta\rangle^{-2-2r} \left(1+(1-2r)|\eta|^2\right) |\widehat{f}|^2 &\\
\lesssim  c_2c_1^{2s} \ \int |\tau|^{2s}\langle\eta\rangle^{-2-2r} \left(1+(1-2r)|\eta|^2\right) |\widehat{f}|^2 &\\
+ c_1^{2s}\|D_v^lf\|_{L_{t,x}^{p_\star}L_v^{q_\star}} \|g\|_{L_{t,x}^{p}L_v^{q}} + c_1^{2s} \|D_t^su\|_{L^{a_1}_{t,x}}\|f\|_{L^{a_2}_{t,x}H_v^{-r}} \|D_t^sf\|_{L^2_{t,x}H^{-r}_v}&.  
\end{aligned}
\end{align}

\medskip\noindent\textit{Step 4. Conclusion.}\\
Combining \eqref{trans-v}, \eqref{trans-x}, \eqref{trans-t} yields that  
\begin{align*}
(1-2r)\int \langle(c_1\tau,\xi)\rangle^{2s}\langle\eta\rangle^{-2r}|\widehat{f}|^2 + \int \langle(c_1\tau,\xi)\rangle^{2s}\langle\eta\rangle^{-2-2r}|\widehat{f}|^2 &\\
\lesssim  c_2c_1^{2s}\int |\tau|^{2s} \langle\eta\rangle^{-2-2r} \left(1+(1-2r)|\eta|^2\right) |\widehat{f}|^2 + \|f\|_{L^2_{t,x}H_v^k}^2 &\\
+ \big(1+\lambda c_1^{-\alpha}\big) \|D_v^lf\|_{L_{t,x}^{p_\star}L_v^{q_\star}} \|g\|_{L_{t,x}^{p}L_v^{q}} + \|D_t^su\|_{L^{a_1}_{t,x}} \|D_t^sf\|_{L^2_{t,x}H^{-r}_v}\|f\|_{L^{a_2}_{t,x}H_v^{-r}}&\\
+\|D_{t,x}^{s_0}u\|_{L_{t,x}^{a_1}} \|D_{t,x}^{2s-s_0}f\|_{L_{t,x}^2H_v^{-r_0}} \|D_v^{1-2r+r_0}f\|_{L_{t,x}^{a_2}L_v^{2}} &.
\end{align*}
By picking $c_2\in(0,1]$ to be sufficiently small, we arrive at the desired result. 
\end{proof}

Theorem~\ref{reg-0} and Theorem~\ref{reg-u} are direct consequences of Proposition~\ref{reg-est}.  
\begin{proof}[Proof of Theorem~\ref{reg-0} and Theorem~\ref{reg-u}]
Parts \ref{reg-1} and \ref{reg-2} of Theorem~\ref{reg-0} are given by Proposition~\ref{reg-est} ($u=0$) with $r=1/2$ and $r=0$, respectively. Theorem~\ref{reg-u} follows from Proposition~\ref{reg-est} ($\lambda=1$) and the Cauchy-Schwarz inequality. Indeed, 
\begin{align*}
\begin{aligned}		
\|D_{t,x}^{s_0}u\|_{L_{t,x}^{a_1}} \|f\|_{L_{t,x}^{a_2}H_v^{1-2r+r_0}} \big(\|D_t^sf\|_{L^2_{t,x}H^{-r}_v} +\|D_{t,x}^{2s-s_0}f\|_{L_{t,x}^2H_v^{-r_0}}\big)\\
\le \varepsilon^{-1}\|D_{t,x}^{s_0}u\|_{L_{t,x}^{a_1}}^2 \|f\|_{L_{t,x}^{a_2}H_v^{1-2r+r_0}}^2 +\varepsilon\|D_t^sf\|_{L^2_{t,x}H^{-r}_v}^2 +\varepsilon\|D_{t,x}^{2s-s_0}f\|_{L_{t,x}^2H_v^{-r_0}}^2&, 
\end{aligned}
\end{align*}
where the constant $\varepsilon:=\varepsilon_0c_*^{2s}(1-2r)$ for some $\varepsilon_0>0$ to be sufficiently small. Since $s_0\in[s,2s]$ and $r_0s+rs_0\ge2rs$, we have the following inequality on the multiplier,  
\begin{align*}
|(\tau,\xi)|^{2s-s_0}\langle\eta\rangle^{-r_0}
&\le \max\big\{1,\,|(\tau,\xi)|^{s}\langle\eta\rangle^{-r_0s/(2s-s_0)}\big\}\\
&\le \max\big\{1,\,|(\tau,\xi)|^{s}\langle\eta\rangle^{-r}\big\}, 
\end{align*}
which implies that 
\begin{align*}
\|D_{t,x}^{2s-s_0}f\|_{L_{t,x}^2H_v^{-r_0}}
\le \|f\|_{L_{t,x,v}^2} + \|D_{t,x}^{s}f\|_{L_{t,x}^2H_v^{-r}}. 
\end{align*}
Combining these two estimates and applying Proposition~\ref{reg-est}, we conclude the proof of Theorem~\ref{reg-u}. 
\end{proof}

In the spirit of \cite{BD}, one can also consider the averaging theory without performing the Fourier transform in time to accommodate time-discretized settings. To elaborate, we consider \eqref{KE} with the constant $\lambda=0$ and vector field $u=u(x)$ depending only on $x$. Instead of \eqref{FKE}, applying the Fourier transform $\mathcal{F}_{x,v}$ to \eqref{KE} yields that
\begin{align*}
X_t\mathcal{F}_{x,v}f=|\xi|^\alpha|\eta|^\beta\mathcal{F}_{x,v}g,\\
X_t:=\partial_t-\xi\cdot\nabla_\eta+i\xi\cdot(\mathcal{F}_xu)*_{\xi}. 
\end{align*}
Using the same arguments as in the proofs presented in this section, particularly setting $c_1=0$ in the proof of Proposition~\ref{reg-est}, leads to the following regularity estimate, which is independent of the localization of solutions. 

\begin{corollary}
Let $(a_1,a_2)=(2,\infty)$ or $(\infty,2)$, $p,q\in(1,\infty)$, $s\in(0,1/2]$, $s_0\in(0,2s]$, $r\in[0,1/2]$, $r_0\in[0,2r]$, and $k,l\ge0$ such that $r_0s+rs_0\ge2rs$ and \eqref{ranges} holds. Assume that the vector field $u=u(x)$ satisfying $D_{x}^{s_0}u\in L^{a_1}(\mathbb{R}^d)$, and the functions $g\in L^{p}_{t,x}L_v^{q}(\mathbb{R}^{1+2d})$ and $f\in L_{t,x}^2H_v^k\cap L_{t,x}^{a_2}H_v^{1-2r+r_0}(\mathbb{R}^{1+2d})$ such that $D_v^lf\in L_{t,x}^{p_\star}L_v^{q_\star}(\mathbb{R}^{1+2d})$. If $f,g$ satisfy \eqref{KE} with $\lambda=0$, then we have 
\begin{align*}
\begin{aligned}
(1-2r)\|D_x^sf\|_{L_{t,x}^2H_v^{-r}}^2 + \|D_x^sf\|_{L_{t,x}^2H_v^{-1-r}}^2&\\
\lesssim \|D_v^lf\|_{L_{t,x}^{p_\star}L_v^{q_\star}} \|g\|_{L_{t,x}^{p}L_v^{q}}
+\|f\|_{L_{t,x}^2H_v^k}^2 
 +(1-2r)^{-1}\|D_{x}^{s_0}u\|_{L_{x}^{a_1}}^2\|f\|_{L_{t,x}^{a_2}H_v^{1-2r+r_0}}^2&. 
\end{aligned}
\end{align*}
\end{corollary}

\section{Endpoint hypoelliptic estimates}\label{endpoint}
This section is dedicated to the proof of Theorem~\ref{reg-p}. To begin, we need to revisit some basic knowledge related to the fractional Kolmogorov equation. By taking advantage of the smoothness and integrability properties of its fundamental solution, we can subsequently formulate some $L^p$ hypoelliptic regularity. 

\subsection{Fractional Kolmogorov equations}
Let us consider the Kolmogorov equation of order $\sigma>0$ with the source term $S=S(t,x,v)$, 
\begin{align}\label{FKol}
\partial_tf+v\cdot\nabla_xf+D_v^\sigma f=S. 
\end{align}
We recall that $D_v^\sigma=(-\Delta_v)^{\sigma/2}$. The fundamental solution of \eqref{FKol} has the form 
\begin{align}\label{G-formula}
G(t,x,v)= \begin{cases}
\ 0 &{\rm\ for\ \;}t\le0,\\
\ \frac{c_{d,\sigma}}{t^{d+2d/\sigma}}\mathcal{G}\left(\frac{x}{t^{1+1/\sigma}},\frac{v}{t^{1/\sigma}}\right) &{\rm\ for\ \;} t>0, 
\end{cases}
\end{align}
so that the solution $f=f(t,x,v)$ to \eqref{FKol} has the representation formula
\begin{align}\label{G-repres}
\begin{aligned}
\int_{-\infty}^t\int_{\mathbb{R}^{2d}} S(t',x',v')\,G(t-t',x-x'-(t-t')v',v-v')\dif t'\dif x'\dif v'. 
\end{aligned}
\end{align}
Here the function $\mathcal{G}$ is nonnegative and can be expressed by means of the Fourier transform $\mathcal{F}_{x,v}$ with respect to $(x,v)\mapsto(\xi,\eta)$,  
\begin{align}\label{GG}
\mathcal{F}_{x,v}\mathcal{G}(\xi,\eta)=e^{-\int_0^1|\vartheta\xi-\eta|^\sigma\dif\vartheta}, 
\end{align}
and the constant $c_{d,\sigma}>0$ depends only on $d,\sigma$ and satisfies $\int_{\mathbb{R}^{2d}}G(t,x,v)\dif x\dif v=1$ for any $t>0$. 

The representation \eqref{G-repres} of solutions in terms of the fundamental solution \eqref{G-formula} can be readily obtained through the method of characteristics following the Fourier transform $\mathcal{F}_{x,v}$ applied to \eqref{FKol}. One may also refer for instance to \cite{Ho,HMP,IS,NZ} for further information on this aspect. 

The precise decay and smoothness properties of the fundamental solution \eqref{G-formula} are related to singular (nonsmooth version) oscillating integrals. The following lemma provides the integrability property $G$ that are sufficient for our purposes. It is a direct consequence of its scaling relation with $\mathcal{G}$ and the corresponding property of $\mathcal{G}$ (see Appendix~\ref{append}). 

\begin{lemma}\label{GG0}
Let $\alpha_0,\beta_0\ge0$ and $p_0,q_0\in[0,\infty]$. There is some constant $C_0>0$ depending only on $d,\sigma,\alpha_0,\beta_0,p_0,q_0$ such that for any $t>0$, 
\begin{align*}
\|D_x^{\alpha_0}D_v^{\beta_0}G(t,x,v)\|_{L_x^{p_0}L_v^{q_0}(\mathbb{R}^{2d})}
\le C_0t^{-\kappa-\alpha_0(1+1/\sigma)-\beta_0/\sigma}, 
\end{align*}
where the constant $\kappa\ge0$ is given by 
\begin{align*}
\kappa:=(d-d/p_0)(1+1/\sigma)+(d-d/q_0)/\sigma. 
\end{align*}
\end{lemma}

\begin{proof}
By considering the formula \eqref{G-formula} for $G$, along with a change of variables that $(x,v)\mapsto(t^{-1-1/\sigma}x,\,t^{-1/\sigma}v)$, we have 
\begin{align*}
\|D_x^{\alpha_0}D_v^{\beta_0}G(t,x,v)\|_{L_x^{p_0}L_v^{q_0}} =c_{d,\sigma} t^{-\kappa-\alpha_0(1+1/\sigma)-\beta_0/\sigma}\|D_x^{\alpha_0}D_v^{\beta_0}\mathcal{G}(x,v)\|_{L_x^{p_0}L_v^{q_0}}. 
\end{align*}
The desired result then follows from part~\ref{GG3} of Lemma~\ref{GGG}. 
\end{proof}

\subsection{Hypoelliptic regularization}
Of concern in this subsection is \eqref{KE} with $\lambda=0$ and $u=0$. Supposing that $f,g,D_v^kf\in L_{t,x}^{p}L_v^{q}(\mathbb{R}^{1+2d})$ for some $p,q\in[1,\infty]$ and $k>0$, we rewrite the equation in the form of \eqref{FKol} with an order $\sigma\in(0,k]$ to be determined, 
\begin{align}\label{FK}
\partial_tf+v\cdot\nabla_xf+D_v^\sigma f = D_x^\alpha D_v^\beta g + D_v^{\sigma-k} D_v^kf . 
\end{align}
We point out that the operator $\partial_t+v\cdot\nabla_x+D_v^\sigma$ is invariant in the sense that it is transformed into $\epsilon^{-\sigma}(\partial_t+v\cdot\nabla_x+D_v^\sigma)$ under the scaling $(t,x,v)\mapsto(\epsilon^\sigma t,\epsilon^{1+\sigma}x,\epsilon v)$ for any fixed $\epsilon>0$. In the light of this scaling relation, we will have to pick the constant $\sigma\ge k$ such that $(1+\sigma)\alpha+\beta = \sigma-k$ so that the two terms on the right-hand side of \eqref{FK} are able to enjoy the same order of regularity. 

Let us now turn to the proof of Theorem~\ref{reg-p}. 
\begin{proof}
In view of \eqref{G-repres} applied to \eqref{FK}, we write 
\begin{align}\label{F12}
D_x^sf=F_1+F_2, 
\end{align} 
where the functions $F_1=F_1(t,x,v)$ and $F_2=F_2(t,x,v)$ are defined by
\begin{align*}
F_1:= \int_{0}^{t+T}\!\!\int_{\mathbb{R}^{2d}} g(t-t',x',v') D_x^s D_{x'}^\alpha D_{v'}^\beta [G(t',x-x'-t'v',v-v')] \dif t'\dif x'\dif v',\\
F_2:= \int_{0}^{t+T}\!\!\int_{\mathbb{R}^{2d}} D_{v'}^kf(t-t',x',v') D_x^s D_{v'}^{\sigma-k}[G(t',x-x'-t'v',v-v')]\dif t'\dif x'\dif v'.
\end{align*}
By Lemma~\ref{GG0}, we see that for any $t>0$, 
\begin{align*}
\|D_x^sD_x^\alpha D_{v'}^\beta[G(t,x-tv',v-v')]\|_{L_x^{p_0}L_v^{q_0}} \lesssim t^{-\kappa-(s+\alpha)(1+1/\sigma)-\beta/\sigma},\\
\|D_x^sD_{v'}^{\sigma-k}[G(t,x-tv',v-v')]\|_{L_x^{p_0}L_v^{q_0}} \lesssim t^{-\kappa-s(1+1/\sigma)-(\sigma-k)/\sigma}. 
\end{align*}
Here the exponents $p_0,q_0\in[1,\infty]$ are defined by
\begin{align*}
1/p'+1=1/p+1/p_0,\quad 1/q'+1=1/q+1/q_0, 
\end{align*}
and the constant $\kappa$ is given in Lemma~\ref{GG0} so that 
\begin{align*}
\kappa=(d/p-d/p')(1+1/\sigma)+(d/q-d/q')/\sigma . 
\end{align*}
It then follows from Minkowski's inequality and Young's inequality that 
\begin{align*}
\|F_1(t,\cdot,\cdot)\|_{L_x^{p'}L_v^{q'}} 
\lesssim \int_0^{t+T} \|g(t-t',\cdot,\cdot)\|_{L_x^{p}L_v^{q}}\, t'^{-\kappa-(s+\alpha)(1+1/\sigma)-\beta/\sigma} \dif t',\\
\|F_2(t,\cdot,\cdot)\|_{L_x^{p'}L_v^{q'}}  
\lesssim \int_0^{t+T} \|D_v^kf(t-t',\cdot,\cdot)\|_{L_x^{p}L_v^{q}}\, t'^{-\kappa-s(1+1/\sigma)-(\sigma-k)/\sigma} \dif t'. 
\end{align*}
Provided that 
\begin{align}\label{reg-expp}
\begin{aligned}
\kappa + (s+\alpha)(1+1/\sigma)+\beta/\sigma<1/p_0,\\
\kappa + s(1+1/\sigma)+(\sigma-k)/\sigma<1/p_0, 
\end{aligned}
\end{align}
then applying Young's inequality again, we have  
\begin{align}\label{F12-est}
\begin{aligned}
\|F_1\|_{L_{t,x}^{p'}L_v^{q'}} \lesssim T^{1/p_0-\kappa-(s+\alpha)(1+1/\sigma)-\beta/\sigma} \|g\|_{L_{t,x}^{p}L_v^{q}},\\
\|F_2\|_{L_{t,x}^{p'}L_v^{q'}} \lesssim T^{1/p_0-\kappa-s(1+1/\sigma)-(\sigma-k)/\sigma}  \|D_v^kf\|_{L_{t,x}^{p}L_v^{q}}. 
\end{aligned}
\end{align}
Gathering these two estimates with \eqref{F12} gives the bound on $\|D_x^sf\|_{L_{t,x}^{p'}L_v^{q'}}$. We now intend to optimize the exponents. To this end, we pick $\sigma$ such that $(1+\sigma)\alpha+\beta = \sigma-k$, meaning that 
\begin{align*}
\sigma=\frac{\alpha+\beta+k}{1-\alpha}, 
\end{align*}
so that \eqref{reg-expp} is equivalent to 
\begin{align*}
(1+\sigma)s + \frac{\sigma+d(1+\sigma)}{p}-\frac{\sigma+d(1+\sigma)}{p'} + \frac{d}{q}-\frac{d}{q'} < k.
\end{align*}
The proof is thus complete. 
\end{proof}

\begin{remark}
If $1<p,q<\infty$ and $(p',q')=(p,q)$, then it has been proven in \cite[Theorem~1.3]{Bouchut} that the solution possesses the same estimate as stated in Theorem~\ref{reg-p}  with the equality $s=(1-\alpha)k/(1+\beta+k)$. 
	
When $p=q=1$, the assumption of the $L_{t,x}^pL_v^q$-norm in Theorem~\ref{reg-p} can be replaced by the total variation of Borel measures. 
\end{remark}

In order to better adapt the estimates to Cauchy problems, we shall consider the solution $f=f(t,x,v)$ to \eqref{FKol} associated with the initial data $f_0=f_0(x,v)$, the source term $S=S(t,x,v)$, that is, 
\begin{align*}
\begin{cases}
\ \partial_tf+v\cdot\nabla_xf+D_v^\sigma f=S &{\rm in\ \;} \mathbb{R}_+\times\mathbb{R}^d\times\mathbb{R}^d, \\
\;\,f|_{t=0}=f_0 &{\rm in\ \;} \mathbb{R}^d\times\mathbb{R}^d. 
\end{cases}
\end{align*}
With its fundamental solution $G$ given by \eqref{G-formula}, the solution $f$ is represented by 
\begin{align}\label{Formula-in}
\begin{aligned}
f(t,x,v)=\int_{\mathbb{R}^{2d}} f_0(x',v') G(t,x-x'-tv',v-v')\dif x'\dif v'\\
+\int_0^t\int_{\mathbb{R}^{2d}} S(t',x',v')\,G(t-t',x-x'-(t-t')v',v-v')\dif t'\dif x'\dif v'&.  
\end{aligned}
\end{align}
Expanding upon the preceding proof of Theorem~\ref{reg-p}, we have the following hypoellipticity result. 

\begin{corollary}\label{reg-cauchy} 
Let $p,q\in[1,\infty]$, $T>0$, and let $f\in L_{t,x}^{p}L_v^{q}([0,T]\times\mathbb{R}^{2d})$ be a solution to \eqref{KE}, with $\lambda=0$, $u=0$, $g\in L^{p}_{t,x}L_v^{q}([0,T]\times\mathbb{R}^{2d})$, satisfying the initial condition $f|_{t=0}=f_0\in L_x^pL_v^{q}(\mathbb{R}^{2d})$. If $D_v^kf\in L_{t,x}^{p}L_v^{q}([0,T]\times\mathbb{R}^{2d})$ for some $k>0$, then for any $p'\ge p$, $q'\ge q$ and $s\ge0$ such that there are some $\omega,\omega'>0$ satisfying 
\begin{align*}
\omega + (1+\sigma)s + \frac{d(1+\sigma)}{p}-\frac{d(1+\sigma)}{p'} + \frac{d}{q}-\frac{d}{q'} = \frac{\sigma}{p'},\\
\omega' + (1+\sigma)s + \frac{\sigma+d(1+\sigma)}{p}-\frac{\sigma+d(1+\sigma)}{p'} + \frac{d}{q}-\frac{d}{q'} = k,
\end{align*}
with $\sigma:=(\alpha+\beta+k)/(1-\alpha)$, we have $D_x^sf\in L_{t,x}^{p'}L_v^{q'}([0,T]\times\mathbb{R}^{2d})$ with  
\begin{align*}
\|D_x^sf\|_{L_{t,x}^{p'}L_v^{q'}([0,T]\times\mathbb{R}^{2d})} \lesssim T^{\;\!\sigma\omega}\|f_0\|_{L_x^pL_v^q(\mathbb{R}^{2d})}\\
+T^{\;\!\sigma\omega'}\|D_v^kf\|_{L_{t,x}^{p}L_v^{q}([0,T]\times\mathbb{R}^{2d})} + T^{\;\!\sigma\omega'}\|g\|_{L_{t,x}^{p}L_v^{q}([0,T]\times\mathbb{R}^{2d})}&.  
\end{align*} 
\end{corollary}

\begin{proof}
Let the constants $\sigma,\kappa,p_0,q_0$ be given in the proof of Theorem~\ref{reg-p} above. Following the argument presented, but instead of using \eqref{F12}, we decompose 
\begin{align*}
D_x^sf=F_0+F_1+F_2. 
\end{align*} 
In view of \eqref{F12} and \eqref{Formula-in}, the additional term $F_0=F_0(t,x,v)$ is defined by
\begin{align*}
F_0:= \int_{\mathbb{R}^{2d}} f_0(x',v') D_x^s[G(t,x-x'-tv',v-v')] \dif x'\dif v'.
\end{align*}
We known from Lemma~\ref{GG0} that for any $t>0$,  
\begin{align*}
\|D_x^s[G(t,x-tv',v-v')]\|_{L_x^{p_0}L_v^{q_0}} \lesssim t^{-\kappa-s(1+1/\sigma)}. 
\end{align*}
By supposing $\kappa + s(1+1/\sigma)<1/p'$, meaning that
\begin{align*}
(1+\sigma)s + \frac{d(1+\sigma)}{p}-\frac{d(1+\sigma)}{p'} + \frac{d}{q}-\frac{d}{q'} < \frac{\sigma}{p'}, 
\end{align*}
we can apply Young's inequality to see that 
\begin{align*}
\|F_0\|_{L_{t,x}^{p'}L_v^{q'}([0,T]\times\mathbb{R}^{2d})} \lesssim T^{1/p'-\kappa-s(1+1/\sigma)} \|f_0\|_{L_x^{p}L_v^{q}(\mathbb{R}^{2d})}.  
\end{align*}
Combining this with \eqref{F12-est} then concludes the proof. 
\end{proof}

\appendix
\section{Fundamental solution of fractional Kolmogorov equation}\label{append}
This appendix demonstrates the basic integrability properties of the fundamental solution of fractional Kolmogorov equation. We point out that the following results are stronger than what we need, but it is worthwhile to present them.  There are also some relevant studies through probabilistic techniques, which can be referenced in \cite{HMP,ChenZhang}. 

\begin{lemma}\label{GGG}
The function $\mathcal{G}$, defined by \eqref{GG} with $\sigma>0$, is smooth and decays at least polynomially at infinity, and has all derivatives of nonnegative order integrable over $\mathbb{R}^{2d}$. To elaborate further, the following properties characterize $\mathcal{G}$. 
\begin{enumerate}[label=(\roman*),leftmargin=0.675cm]
\item\label{GG1} There is some constant $c_\sigma>0$ depending only $\sigma$ such that for any $(\xi,\eta)\in\mathbb{R}^{2d}$, 
\begin{align*}
\mathcal{F}_{x,v}\mathcal{G}(\xi,\eta)\lesssim e^{-c_\sigma|(\xi,\eta)|^\sigma}. 
\end{align*}
In particular, $D_{x,v}^{m_0}\mathcal{G}$ with $m_0\ge0$ is bounded in $L^\infty(\mathbb{R}^{2d})$. 
\item\label{GG2} For any $a\in[0,\sigma)$ and $m\in\mathbb{N}^{2d}$, we have 
\begin{align*}
\|\langle(x,v)\rangle^a\partial_{x,v}^m\mathcal{G}\|_{L^1(\mathbb{R}^{2d})}\lesssim_{a,m} 1. 
\end{align*}
In particular, $D_{x,v}^{m_0}\mathcal{G}$ with $m_0\ge0$ is bounded in $L^1(\mathbb{R}^{2d})$. 
\item\label{GG3} For any $m_0\ge0$ and $p_0,q_0\in[0,\infty]$, we have 
\begin{align*}
\|D_{x,v}^{m_0}\mathcal{G}\|_{L_x^{p_0}L_v^{q_0}(\mathbb{R}^{2d})}\lesssim_{m_0,p_0,q_0} 1. 
\end{align*}
\end{enumerate}
\end{lemma}

\begin{proof}
For part~\ref{GG1}, we fix $\xi\in\mathbb{R}^d$ and consider the function $\int_0^1|\vartheta\xi-\eta|^\sigma\dif\vartheta$ in the variable $\eta\in\mathbb{R}^d$. It attains its minimum at $\eta=\xi/2$ so that  
\begin{align*}
2^{\sigma}(1+\sigma)\int_0^1|\vartheta\xi-\eta|^\sigma\dif\vartheta\ge|\xi|^\sigma. 
\end{align*} 
It also follows that 
\begin{align*}
2^\sigma[1+2^{\sigma}(1+\sigma)]\int_0^1|\vartheta\xi-\eta|^\sigma\dif\vartheta \ge 2^\sigma\int_0^1(|\vartheta\xi-\eta|^\sigma+|\xi|^\sigma)\dif\vartheta \ge |\eta|^\sigma. 
\end{align*} 
We hence derive the exponential decay for $\mathcal{F}_{x,v}\mathcal{G}$. 

For part~\ref{GG2}, by abbreviating $w:=(x,v)$, $\zeta:=(\xi,\eta)$, we first have to show that 
\begin{align*}
\left|\int_{\mathbb{R}^{4d}}\langle w\rangle^a e^{iw\cdot\zeta}\mathcal{F}_w\mathcal{G}(\zeta)\dif\zeta\!\dif w\right|\lesssim_{a}1. 
\end{align*}
We know from part~\ref{GG1} that it suffices to establish that for some $n_0>0$, 
\begin{align}\label{GGD}
\left|\int_{\mathbb{R}^{4d}}\frac{|w|^{a+n_0}}{\langle w\rangle^{n_0}} e^{iw\cdot\zeta}\mathcal{F}_w\mathcal{G}(\zeta)\dif\zeta\!\dif w\right|\lesssim_{a}1. 
\end{align}
Consider the two operators 
\begin{align*}
\mathcal{L}_1:=-i|w|^{-2}w\cdot\nabla_\zeta,\quad \mathcal{L}_2:=-i|\zeta|^{-2}\zeta\cdot\nabla_w.
\end{align*}
For $l=1,2$, $\mathcal{L}_l$ satisfies $\mathcal{L}_l(e^{iw\cdot\zeta})=e^{iw\cdot\zeta}$ and has the adjoint $-\mathcal{L}_l$. We thus have 
\begin{align*}
\int_{\mathbb{R}^{2d}} e^{iw\cdot\zeta}\mathcal{F}_w\mathcal{G}\dif\zeta =-\int_{\mathbb{R}^{2d}} e^{iw\cdot\zeta}\mathcal{L}_1\mathcal{F}_w\mathcal{G}\dif\zeta, 
\end{align*} 
where 
\begin{align*}
\mathcal{L}_1\mathcal{F}_w\mathcal{G}=i\sigma|w|^{-2}\mathcal{F}_w\mathcal{G}\int_0^1w\cdot(\vartheta^2\xi-\vartheta\eta,\eta-\vartheta\xi)\,|\vartheta\xi-\eta|^{\sigma-2}\dif\vartheta&
\end{align*} 
so that it satisfies  
\begin{align}\label{LFG}
\left|\mathcal{L}_1\mathcal{F}_w\mathcal{G}\right|\lesssim |w|^{-1}|\zeta|^{\sigma-1}\mathcal{F}_w\mathcal{G}.
\end{align}
Let us pick a function $\chi_1\in C^\infty(\mathbb{R}_w^{2d}\times\mathbb{R}_\zeta^{2d})$ valued in $[0,1]$ such that $\chi_1=1$ for $2\langle w\rangle\le|\zeta|^{-1}$, and $\chi_1=0$ for $\langle w\rangle\ge|\zeta|^{-1}$, and $|\partial_w^k\chi_1|\lesssim_k \langle w\rangle^{-|k|}$ for any $k\in\mathbb{N}^d$. By setting $\chi_2:=1-\chi_1$, we decompose the integral in \eqref{GGD} as follows, 
\begin{align}\label{GGD0}
\int_{\mathbb{R}^{4d}} \frac{|w|^{a+n_0}}{\langle w\rangle^{n_0}} e^{iw\cdot\zeta}\mathcal{F}_w\mathcal{G}\dif\zeta\!\dif w =- \int_{\mathbb{R}^{4d}}(\chi_1+\chi_2) \frac{|w|^{a+n_0}}{\langle w\rangle^{n_0}}e^{iw\cdot\zeta}\mathcal{L}_1\mathcal{F}_w\mathcal{G}\dif w\!\dif\zeta.
\end{align} 
By the property of $\chi_1$, we deduce that for a fixed constant $a'\in(a,\sigma)$, 
\begin{align*}
\left|\int_{\mathbb{R}^{4d}} \chi_1\langle w\rangle^ae^{iw\cdot\zeta}\mathcal{L}_1\mathcal{F}_w\mathcal{G}\dif w\!\dif\zeta\right| \le \int_{\mathbb{R}^{4d}}\langle w\rangle^{-2d-a'+1+a}|\zeta|^{-2d-a'+1}\left|\mathcal{L}_1\mathcal{F}_w\mathcal{G}\right|\dif w\!\dif\zeta. 
\end{align*} 
It then follows from the results of \eqref{LFG} and part~\ref{GG1} that
\begin{align}\label{GGD1}
\begin{aligned}
\left|\int_{\mathbb{R}^{4d}} \chi_1\langle w\rangle^ae^{iw\cdot\zeta}\mathcal{L}_1\mathcal{F}_w\mathcal{G}\dif w\!\dif\zeta\right| &\lesssim \int_{\mathbb{R}^{4d}} \langle w\rangle^{-2d-a'+a}|\zeta|^{-2d-a'+\sigma}\mathcal{F}_w\mathcal{G}\dif w\!\dif\zeta\\
&\lesssim_a \int_{\mathbb{R}^{2d}} |\zeta|^{-2d-a'+\sigma}\mathcal{F}_w\mathcal{G}\dif\zeta
\lesssim_a 1. 
\end{aligned} 
\end{align} 
In view of the property of $\mathcal{L}_2$, we know that for any $n\in\mathbb{Z}_+$, 
\begin{align*}
\int_{\mathbb{R}^{4d}} \chi_2\frac{|w|^{a+n_0}}{\langle w\rangle^{n_0}}e^{iw\cdot\zeta}\mathcal{L}_1\mathcal{F}_w\mathcal{G}\dif w\!\dif\zeta = \int_{\mathbb{R}^{4d}} e^{iw\cdot\zeta}(-\mathcal{L}_2)^n\!\left[\chi_2\frac{|w|^{a+n_0}}{\langle w\rangle^{n_0}}\mathcal{L}_1\mathcal{F}_w\mathcal{G}\right]\dif w\!\dif\zeta. 
\end{align*} 
Based on \eqref{LFG}, a direct computation shows that for any $n\in\mathbb{Z}_+$, 
\begin{align*}
\left|\mathcal{L}_2^n\left[\chi_2|w|^{a+n_0}\langle w\rangle^{-n_0}\mathcal{L}_1\mathcal{F}_w\mathcal{G}\right]\right|
\lesssim_{n} \langle w\rangle^{a-n_0}|w|^{n_0-1-n}|\zeta|^{\sigma-1-n}\mathcal{F}_w\mathcal{G}. 
\end{align*} 
We may fix the constant $n\in\mathbb{Z}_+$ such that $n\ge 2d+\sigma$ and set $n_0:=n+1$ to see from the above estimate and the property of $\chi_2$ that for any fixed $a'\in(a,\sigma)$,
\begin{align*}
\left|\int_{\mathbb{R}^{4d}} \chi_2\frac{|w|^{a+n_0}}{\langle w\rangle^{n_0}}e^{iw\cdot\zeta}\mathcal{L}_1\mathcal{F}_w\mathcal{G}\dif w\!\dif\zeta\right| \lesssim \int_{2\langle w\rangle\ge|\zeta|^{-1}}\langle w\rangle^{a-n_0}|\zeta|^{\sigma-n_0}\mathcal{F}_w\mathcal{G}\dif w\!\dif\zeta\\
\lesssim \int_{\mathbb{R}^{4d}}\langle w\rangle^{a-a'-2d}|\zeta|^{\sigma-a'-2d}\mathcal{F}_w\mathcal{G}\dif w\!\dif\zeta \lesssim_a 1&. 
\end{align*} 
Combining this with \eqref{GGD0} and \eqref{GGD1}, we obtain \eqref{GGD} as claimed, which then implies part~\ref{GG2} with $|m|=0$. The result for general $m\in\mathbb{N}^d$ can be derived through interpolation. Indeed, we know from part~\ref{GG1} that $\partial_w^m(\langle w\rangle^a\mathcal{G})$ is bounded in $L^\infty(\mathbb{R}^{2d})$ for all $m\in\mathbb{N}^d$. By the Gagliardo-Nirenberg interpolation and the fact that $\langle w\rangle^a\mathcal{G}\in L^1(\mathbb{R}^{2d})$, we derive the boundedness of $\partial_w^m(\langle w\rangle^a\mathcal{G})\in L^1(\mathbb{R}^{2d})$ for all $m\in\mathbb{N}^d$. It then follows from induction and the triangle inequality that part~\ref{GG2} holds for any $m\in\mathbb{N}^{2d}$. 

Finally, the result of part~\ref{GG3} is a direct consequence of interpolation and the boundedness of $D_w^{m_0}\mathcal{G}\in L^\infty\cap L^1(\mathbb{R}^{2d})$ for all $m_0\ge 0$. 
\end{proof}

\end{document}